\documentclass[11pt]{amsart}

\usepackage[USenglish,english]{babel}
\usepackage{amsfonts}
\usepackage{amsmath}
\usepackage{amssymb}
\usepackage{amsthm}
\usepackage[T1]{fontenc}
\usepackage[latin1]{inputenc}
\usepackage{enumerate}
\usepackage{graphicx}
\usepackage[all]{xy}
\usepackage{accents}
\usepackage{mathtools}
\usepackage{mathrsfs}
\usepackage{comment}
\usepackage{accents}
\usepackage{tikz}
\usetikzlibrary{matrix,arrows,decorations.pathmorphing}
\usepackage{tikz-cd}
\usepackage{afterpage}
\usepackage{bbm}
\usepackage{dsfont}
\usepackage{stmaryrd}
\usepackage{hyperref}
\usepackage{mleftright}
\usepackage{subfig}
\usepackage{faktor}  

\theoremstyle{plain}
\newtheorem{thm}{Theorem}[section]
\newtheorem{lem}[thm]{Lemma}
\newtheorem{prop}[thm]{Proposition}
\newtheorem{cor}[thm]{Corollary}
\newtheorem*{thm*}{Theorem}
\newtheorem*{cor*}{Corollary}
\newtheorem{thmintro}{Theorem}

\newtheorem{propintro}[thmintro]{Proposition}

\theoremstyle{definition}
\newtheorem{defn}[thm]{Definition}
\newtheorem*{dfn*}{Definition}
\newtheorem{ex}[thm]{Example}
\newtheorem{rmk}[thm]{Remark}

\renewcommand{\o}{\circ}
\newcommand{\wt}{\widetilde}

\newcommand{\R}{\mathbb{R}}

\newcommand{\N}{\mathbb{N}}

\newcommand{\s}{\sigma}

\newcommand{\ra}{\rightarrow}
\newcommand{\cu}{\subseteq}

\newcommand{\G}{\Gamma}

\newcommand{\mc}{\mathcal}
\newcommand{\mf}{\mathfrak}
\newcommand{\x}{\times}

\newcommand{\eps}{\epsilon}
\newcommand{\Om}{\Omega}

\newcommand{\wh}{\widehat}
\newcommand{\mscr}{\mathscr}

\begin{document}

\title{Roller boundaries for median spaces and algebras}
\author{Elia Fioravanti}

\begin{abstract}
We construct compactifications for median spaces with compact intervals, generalising Roller boundaries of ${\rm CAT}(0)$ cube complexes. Examples of median spaces with compact intervals include all finite rank median spaces and all proper median spaces of infinite rank. Our methods also apply to general median algebras, where we recover the zero-completions of \cite{Bandelt-Meletiou}. Along the way, we prove various properties of halfspaces in finite rank median spaces and a duality result for locally convex median spaces.
\end{abstract}


\maketitle

\tableofcontents

\section{Introduction.}

The aim of this paper is to construct a compactification with good median properties for certain classes of median algebras and median spaces. Median algebras were originally introduced in order theory as a common generalisation of dendrites and lattices; they have been extensively studied in relation to semi-lattices (see e.g.\ \cite{Sholander,Isbell,Bandelt-Hedlikova}) and, more recently, in more geometrical terms because of their connections to ${\rm CAT}(0)$ cube complexes and median spaces (for instance in \cite{Roller,Bow1,Bow2}). 

A metric space $X$ is said to be \emph{median} if, for any three points $x_1,x_2,x_3$ of $X$, there exists a unique \emph{median}, i.e.\ a unique point $m=m(x_1,x_2,x_3)\in X$ such that $d(x_i,x_j)=d(x_i,m)+d(m,x_j)$ for all $1\leq i<j\leq 3$. In this case, we refer to the induced map $m\colon X^3\ra X$ as the \emph{median map}. The $0$-skeleton of any ${\rm CAT}(0)$ cube complex becomes a median metric space if we endow it with the restriction of the intrinsic path metric of the $1$-skeleton. More generally, every real tree and every ultralimit of median spaces is median. The latter include, in particular, all asymptotic cones of cube complexes.

Further examples of median spaces arise from Guirardel cores of pairs of actions on real trees \cite{Guirardel} and from asymptotic cones of coarse median groups of finite rank \cite{Bow1,Zeidler}. Examples of the latter are provided by mapping class groups, cubulated groups and most irreducible 3-manifold groups. Finally, we remark that $L^1(X,\mu)$ is a median space for any measure space $(X,\mu)$. For additional literature on the subject, see for instance \cite{Vel, Verheul, Nica-thesis, CDH, Bow4} and references therein.

The theory of median spaces has two essentially distinct flavours. On the one hand, the study of infinite dimensional median spaces is strongly related to functional analysis. For instance, a locally compact, second countable group admits a (metrically) proper action on a median space if and only if it has the Haagerup property. Similarly, Kazhdan groups are precisely those that can act on median spaces only with bounded orbits \cite{Cherix-Martin-Valette,CDH}. 

On the other hand, the study of finite dimensional median spaces (\emph{finite rank}\footnote{For connected median spaces, the \emph{rank} can be defined as the supremum of the topological dimensions of locally compact subsets; see Section~\ref{median algebras}.} in our terminology) strongly resembles that of ${\rm CAT}(0)$ cube complexes, with the additional pathologies typical of real trees.

A key feature of cube complexes is that they come equipped with a collection of hyperplanes. These give each ${\rm CAT}(0)$ cube complex a canonical structure of \emph{space with walls} \cite{Haglund-Paulin}. Conversely, every space with walls can be canonically embedded into a ${\rm CAT}(0)$ cube complex \cite{Sageev, Nica, Chatterji-Niblo}. In fact, the relationship between spaces with walls and cube complexes can be viewed as a form of duality; see Corollary~4.10 in \cite{Nica}. A similar phenomenon arises for more general median spaces, as we now describe.

Spaces with measured walls were introduced in \cite{Cherix-Martin-Valette}; they provide useful characterisations for the Haagerup and Kazhdan properties \cite{Cherix-Martin-Valette, Cornulier-Tessera-Valette, CDH}. Each space with measured walls $Z$ can be canonically embedded into a median space $\mc{M}(Z)$ called its \emph{medianisation}. Conversely, to every median space $X$ there corresponds a canonical set of walls, namely its \emph{convex walls}; this induces a structure of space with measured walls on $X$ such that $X\hookrightarrow\mc{M}(X)$ \cite{CDH}. 

We prove the following analogue of the duality result available for cube complexes (Corollary~4.10 in \cite{Nica}). In particular, this applies to all complete, finite rank median spaces:

\begin{thmintro}\label{A}
For every complete, locally convex median space $X$, the inclusion $X\hookrightarrow\mc{M}(X)$ is a surjective isometry.
\end{thmintro}

As in cube complexes, walls split every median space into \emph{halfspaces}. In general, the behaviour of halfspaces can be extremely complicated. For instance, if $(X,\mu)$ is a standard probability space with no atoms, every halfspace of $L^1(X,\mu)$ is dense (Example~\ref{dense halfspaces in L^1}). This however does not happen in finite rank median spaces:

\begin{propintro}\label{B}
In a complete, finite rank median space every halfspace is either open or closed. If $\mf{h}_1\supsetneq ... \supsetneq\mf{h}_k$ is a chain of halfspaces with $k>2\cdot\text{rank}(X)$, the closure $\overline{\mf{h}_k}$ is contained in the interior of $\mf{h}_1$.
\end{propintro}

Many of the analogies between median spaces and ${\rm CAT}(0)$ cube complexes resemble those between real trees and simplicial trees. It is thus natural to wonder (see e.g.\ Question~1.11 in \cite{CDH}) whether a group acting on a median space with unbounded orbits (resp.\ properly) must have an action on a ${\rm CAT}(0)$ cube complex with unbounded orbits (resp.\ proper). 

Both questions have a negative answer. For instance, irreducible lattices in $O(4,1;\R)\x O(3,2;\R)$ do not have property (T), but all their actions on ${\rm CAT}(0)$ cube complexes fix a point; see Theorem~6.14 in \cite{Cor2}. Furthermore, Baumslag-Solitar groups $BS(m,n)$ with $m\neq n$ have the Haagerup property, but do not act properly on any ${\rm CAT}(0)$ cube complex \cite{Haglund}. 

We are however unable to answer the previous questions under the additional assumption that the median space be of \emph{finite rank}. Even for groups acting on real trees, it is a delicate matter \cite{Minasyan}. See \cite{Chatterji-Drutu} for a discussion of similar problems.

In the present paper, we bring the analogies between finite dimensional ${\rm CAT}(0)$ cube complexes and finite rank median spaces one step further, by extending to median spaces the construction of the Roller compactification.

Roller boundaries of ${\rm CAT}(0)$ cube complexes are implicit in \cite{Roller}, although the definition that is most commonly used today probably first appeared in \cite{BCGNW}. They have been profitably used to obtain various interesting results, for instance, without attempting to be exhaustive, in \cite{BCGNW, Nevo-Sageev,CFI,Fernos,Fernos-Lecureux-Matheus}. 

It is well-known that Roller boundaries of cube complexes can be given the following two equivalent characterisations. We denote by $\mscr{H}$ the set of halfspaces of the cube complex $X$ and do not distinguish between $X$ and its $0$-skeleton. 

\begin{enumerate}
\item We can embed $X$ into $2^{\mscr{H}}$ by mapping each vertex $v$ to the set $\s_v:=\{\mf{h}\in\mscr{H}\mid v\in\mf{h}\}$. The space $2^{\mscr{H}}$ is compact with the product topology. Thus, the closure $\overline X$ of $X$ inside $2^{\mscr{H}}$ is compact; we refer to it as the \emph{Roller compactification}. The Roller boundary is the set $\partial X:=\overline X\setminus X$. 
\item An \emph{ultrafilter} on $\mscr{H}$ is a maximal subset $\s\cu\mscr{H}$ such that any two halfspaces in $\s$ intersect. The Roller compactification $\overline X$ coincides with the subset of $2^{\mscr{H}}$ consisting of all ultrafilters. Boundary points correspond to ultrafilters that contain infinite descending chains of halfspaces.
\end{enumerate}

It should be noted that a different definition of the Roller boundary appears in \cite{Guralnik}; see \cite{Genevois} for a discussion of this alternative notion. 

For a general median space $X$, we will give four equivalent definitions of the Roller compactification. We sketch them here to illustrate the issues that arise when leaving the discrete world of cube complexes.

\begin{enumerate}
\item As in cube complexes, we denote the set of halfspaces by $\mscr{H}$. In principle, one could try to define a compactification as we did above, namely by taking the closure of the image of $X\hookrightarrow 2^{\mscr{H}}$. However, if $X$ is not discrete, this results in a space that is too large and carries little geometrical meaning: the double dual $X^{\o\o}$ \cite{Roller}. See Remark~\ref{bad ultrafilters on R} and Example~\ref{double dual of N} for the pathologies that may arise; here we simply remark that the inclusion $X\hookrightarrow X^{\o\o}$ needs not be continuous.

Instead, given $x,y\in X$, we denote by $I(x,y)$ the union of all medians $m(x,y,z)$ with $z\in X$; we refer to $I(x,y)$ as the \emph{interval} between $x$ and $y$. In many interesting cases, intervals are compact. 
We define the Roller compactification $\overline X$ as the closure of the image of the map:
\begin{align*}
X & \hookrightarrow \prod_{x,y\in X} I(x,y) \\
z & \mapsto (m(x,y,z))_{x,y},
\end{align*}
and we set $\partial X:=\overline X\setminus X$. For ${\rm CAT}(0)$ cube complexes, this provides a new characterisation of the customary Roller boundary. A similar construction was considered in \cite{Ward} for dendrites.
\item As in cube complexes, we could try to consider the set of ultrafilters on $\mscr{H}$; this would however result again in the double dual $X^{\o\o}$. Instead, we endow $\mscr{H}$ with a $\s$-algebra of subsets and a measure; we need a finer $\s$-algebra than the one in \cite{CDH}, see Section~\ref{measure theory}. We then only consider measurable ultrafilters, identifying sets with null symmetric difference. This is an alternative description of the Roller compactification $\overline X$ (see Theorem~\ref{to highlight}), although no natural topology arises via this construction. 
\item Given a general median algebra $M$, the \emph{zero-completion} of $M$ was introduced in \cite{Bandelt-Meletiou}. We will recover the same object from a more geometrical perspective in Section~\ref{the zero-completion}. When $M$ is the median algebra arising from the median space $X$, the zero-completion of $M$ is identified with the Roller compactification $\overline X$. This shows that $\overline X$ has itself a natural structure of median algebra.
\item Finally, $\overline X$ can be naturally identified with the horofunction compactification of $X$. For $0$-skeleta of ${\rm CAT}(0)$ cube complexes, this is an unpublished result of U.\ Bader and D.\ Guralnik; also see \cite{Caprace-Lecureux,Fernos-Lecureux-Matheus}.
\end{enumerate}

Each of the four definitions is better suited to the study of a particular aspect of $\overline X$. Their interplay yields:

\begin{thmintro}\label{C}
Let $X$ be a complete, locally convex median space with compact intervals. The Roller compactification $\overline X$ is a compact, locally convex, topological median algebra. The inclusion $X\hookrightarrow\overline X$ is a continuous morphism of median algebras, whose image is convex and dense. It is an embedding if, in addition, $X$ is connected and locally compact.
\end{thmintro}

The class of locally convex median spaces with compact intervals encompasses both complete, finite rank median spaces and (possibly infinite dimensional) ${\rm CAT}(0)$ cube complexes. For the latter, we recover the usual Roller compactification. We also remark that every complete, connected, locally compact median space is proper\footnote{See Lemma~4.6 in \cite{Bow4} and Proposition~I.3.7 in \cite{BH}.} and thus has compact intervals; examples of such spaces that are locally convex and infinite-rank appear e.g.\ in \cite{Chatterji-Drutu}. 

As for ${\rm CAT}(0)$ cube complexes, Roller compactifications of median spaces are endowed with an extended metric $d\colon\overline X\x\overline X\ra [0,+\infty]$. Thus, they are partitioned into \emph{components}, namely maximal subsets of points, any two of which are at finite distance. The space $X$ always forms an entire component; moreover:

\begin{thmintro}\label{D}
Let $X$ be complete and finite rank. Every component of $\partial X$ is a complete median space of strictly lower rank.
\end{thmintro}

In subsequent work, Theorem~\ref{D} will allow us to prove a number of results by induction on the rank. In \cite{Fioravanti2}, we use Roller boundaries to extend to finite rank median spaces the machinery developed in \cite{CS} and part of Hagen's theory of unidirectional boundary sets (UBS) \cite{Hagen}. As a consequence, we obtain in \cite{Fioravanti2} a version of the Tits alternative for groups acting freely on finite rank median spaces. 

In \cite{Fioravanti3}, we generalise to finite rank median spaces the superrigidity result of \cite{CFI}. As a consequence, if $\G$ is an irreducible lattice in a higher rank semisimple Lie group, every action of $\G$ on a complete, connected, finite rank median space must fix a point. 

This is in sharp contrast to the behaviour of actions on infinite rank median spaces. Indeed, it was shown in \cite{Chatterji-Drutu} that, for $k,n\geq 2$, all uniform lattices in $(PSL_k\R)^n$ admit proper, cocompact\footnote{To be precise, \cite{Chatterji-Drutu} only proves that the action is \emph{cobounded}, but it is possible to show that the target median space is proper. The authors have informed me that this stronger result (implying cocompactness) will appear in the next version of their preprint.} actions on complete, connected, infinite-rank median spaces. Note that this phenomenon is specific to non-discrete median spaces, as every cocompact cube complex is finite dimensional.

{\bf Structure of the paper.} In Section~\ref{main characters} we give definitions and basic results. We study convexity, intervals and halfspaces; we prove Proposition~\ref{B}. In Section~\ref{measure theory}, we construct a $\s$-algebra and a measure on the set of halfspaces of a median space; we also prove Theorem~\ref{A}. In Section~\ref{the zero-completion} we study zero-completions of median algebras; our perspective is different from the one in \cite{Bandelt-Meletiou}, but we show that our notions are equivalent. Section~\ref{definition and first properties} is devoted to Roller compactifications of median spaces; we prove Theorem~\ref{C} there. Finally, we analyse components in Section~\ref{cc's} and prove Theorem~\ref{D}.

{\bf Acknowledgements.} The author warmly thanks Brian Bowditch, Pier\-re-Emmanuel Caprace, Indira Chatterji, Yves Cornulier, Talia Fern\'os and Mark Hagen for many helpful conversations and Anthony Genevois for his comments on an earlier version. The author expresses special gratitude to Cornelia Dru\c tu for her excellent supervision and to Talia Fern\'os for her encouragement to pursue this project. We also thank the anonymous referee for many valuable comments.

This work was undertaken at the Mathematical Sciences Research Institute in Berkeley during the Fall 2016 program in Geometric Group Theory, where the author was supported by the National Science Foundation under Grant no.\ DMS-1440140 and by the GEAR Network. Part of this work was also carried out at the Isaac Newton Institute for Mathematical Sciences, Cambridge, during the programme ``Non-positive curvature, group actions and cohomology'' and was supported by EPSRC grant no.\ EP/K032208/1. The author was also supported by the Clarendon Fund and the Moussouris Scholarship of Merton College.

\section{Preliminaries.}\label{main characters}

\subsection{Median algebras and median spaces.}\label{median algebras}

For definitions and various results on median algebras from a geometric perspective, we refer the reader to \cite{CDH, Bow1, Bow4, Roller}. In the following discussion we consider a median algebra $M$ with median map $m$.

Given $x,y\in M$ the \emph{interval} $I(x,y)$ is the set of points $z\in M$ satisfying $m(x,y,z)=z$. A finite or infinite sequence of points $x_k\in M$ is a \emph{discrete geodesic} if $x_k\in I(x_m,x_n)$ whenever $m<k<n$. A subset $C\cu M$ is said to be \emph{convex} if $m(x,y,z)\in C$ whenever $x,y\in C$ and $z\in M$; equivalently $I(x,y)\cu C$ for all $x,y\in C$. Any collection of pairwise-intersecting convex subsets has the finite intersection property; this fact is usually known as Helly's Theorem, see e.g.\ Theorem~2.2 in \cite{Roller}.

A \emph{convex halfspace} is a nonempty convex subset $\mf{h}\cu M$ whose complement $\mf{h}^*:=M\setminus\mf{h}$ is also convex and nonempty. A \emph{convex wall} is an unordered pair $\mf{w}:=\{\mf{h}, \mf{h}^*\}$, where $\mf{h}$ is a convex halfspace; we will generally simply speak of \emph{halfspaces} and \emph{walls}, unless we need to avoid confusion with the notions in Section~\ref{SMW}. The wall $\mf{w}$ \emph{separates} subsets $A,B$ of $M$ if $A\cu\mf{h}$ and $B\cu\mf{h}^*$ or vice versa. Any two disjoint convex subsets can be separated by a wall, see e.g.\ Theorem~2.7 in \cite{Roller}.

We will denote the collections of all walls and all halfspaces of $M$ by $\mscr{W}(M)$ and $\mscr{H}(M)$, respectively, or simply by $\mscr{W}$ and $\mscr{H}$ when the context is clear; there is a natural two-to-one projection $\pi\colon\mscr{H}\ra\mscr{W}$. Given subsets $A,B\cu M$, we write $\mscr{W}(A|B)$ for the set of walls separating them and set
\[\mscr{H}(A|B):=\{\mf{h}\in\mscr{H}\mid B\cu\mf{h}, \ A\cu\mf{h}^*\}.\]
We will simply write $\s_A$ for $\mscr{H}(\emptyset|A)$ and confuse the singleton $\{x\}$ with the point $x$. We will refer to sets of the form $\mscr{W}(x|y)$ and $\mscr{H}(x|y)$ as \emph{wall-intervals} and \emph{halfspace-intervals}, respectively.

A \emph{pocset} $(\mc{P},\preceq,*)$ consists of a poset $(\mc{P},\preceq)$ equipped with an order-reversing involution $*$, such that every element $a\in\mc{P}$ is incomparable with $a^*$. Elements $a,b$ of a pocset are \emph{transverse} if any two elements of the set $\{a,a^*,b,b^*\}$ are incomparable. Considering the median algebra $M$, the triple $(\mscr{H},\cu,*)$ is a pocset. Halfspaces $\mf{h},\mf{k}$ are transverse if and only if all the intersections $\mf{h}\cap\mf{k}$, $\mf{h}\cap\mf{k}^*$, $\mf{h}^*\cap\mf{k}$, $\mf{h}^*\cap\mf{k}^*$ are nonempty. We say that two walls are transverse if they correspond to transverse halfspaces. 

A \emph{partial filter} is a subset $\s\cu\mc{P}$ such that there do not exist $a,b\in\s$ with $a\preceq b^*$. If moreover for every $a\in\mc{P}$ we either have $a\in\s$ or $a^*\in\s$, we say that $\s$ is an \emph{ultrafilter}. A \emph{filter} is a partial filter $\s$ such that $b\in\s$ whenever $a\preceq b$ and $a\in\s$. Some care is needed when comparing our terminology to that of \cite{CFI}, as their notion of ``partially defined ultrafilter'' coincides with our notion of ``filter''.

Every partial filter $\s\cu\mc{P}$ is contained in a filter; the smallest such filter is the set of all $b\in\mc{P}$ such that $a\preceq b$ for some $a\in\s$. Every filter is contained in an ultrafilter; in fact, ultrafilters are precisely filters that are maximal under inclusion. Given any subset $A\cu M$, the set $\s_A\cu\mscr{H}$ is a filter; it is an ultrafilter if and only if $A$ consists of a single point.

A subset $\s\cu\mc{P}$ is said to be \emph{inseparable} if, whenever $a\preceq b\preceq c$ and $a,c\in\s$, we also have $b\in\s$. Given a subset $\s\cu\mscr{P}$, its \emph{inseparable closure} is the smallest inseparable subset of $\mc{P}$ containing $\s$. It consists precisely of all those $b\in\mc{P}$ such that there exist $a,c\in\s$ with $a\preceq b\preceq c$. Note that the inseparable closure of a partial filter is again a partial filter; all filters are inseparable.

The set $\{-1,1\}$ has a unique structure of median algebra. If ${k\in\N}$, a \hbox{\emph{$k$-hypercube}} is the median algebra $\{-1,1\}^k$, given by considering the median map of $\{-1,1\}$ separately in all coordinates. The \emph{rank} of the median algebra $M$ is the maximal $k\in\N$ such that we can embed a $k$-hypercube into $M$; if $M$ has at least two points, we have rank$(M)\in [1,+\infty]$. The rank of $M$ coincides with the maximal cardinality of a set of pairwise-transverse halfspaces, see Proposition~6.2 in \cite{Bow1}.

Given a subset $C\cu M$ and $x\in M$, $y\in C$, we say that $y$ is a \emph{gate} for $(x,C)$ if $y\in I(x,z)$ for every $z\in C$. We say that $C\cu M$ is \emph{gate-convex} if a gate for $(x,C)$ exists for every $x\in M$. If $C$ is gate-convex, there is a unique gate for $(x,C)$ for every $x\in M$; thus we can define a \emph{gate-projection} $\pi_C\colon M\ra C$. Gate-convex subsets are always convex, but the converse is not true in general; see Lemma~\ref{gate-convexity vs compact intervals} for an obstruction. The interval $I(x,y)$ is always gate-convex with gate-projection given by $\pi(z)=m(x,y,z)$; on the other hand, if $C\cu I(x,y)$ is gate-convex, we have $C=I(\pi_C(x),\pi_C(y))$.

\begin{prop}\label{retractions new}
A map $\phi\colon M\ra M$ is a gate-projection to its image if and only if, for all $x,y,z\in M$, we have $\phi(m(x,y,z))=m(\phi(x),\phi(y),z)$. In this case, we also have $\phi\left(m(x,y,z)\right)=m\left(\phi(x),\phi(y),\phi(z)\right)$; in particular, gate-projections map intervals to intervals.
\end{prop}
\begin{proof}
See Proposition~5.1 in \cite{Bandelt-Hedlikova} and 5.8 in \cite{Isbell}. Note that ``retract'' and ``\v Ceby\v sev ideals'' are alternative terminology for ``gate-convex subset''; Isbell works in the more general context of isotropic media. 
\end{proof}

\begin{lem}\label{gates vs inclusions}
\begin{enumerate}
\item If $C_1\cu M$ is convex and $C_2\cu M$ is gate-convex, the projection $\pi_{C_2}(C_1)$ is convex. If moreover $C_1\cap C_2\neq\emptyset$, we have $\pi_{C_2}(C_1)=C_1\cap C_2$.
\item If $C_1,C_2\cu M$ are gate-convex, the sets $\pi_{C_1}(C_2)$ and $\pi_{C_2}(C_1)$ are gate-convex with gate-projections $\pi_{C_1}\o\pi_{C_2}$ and $\pi_{C_2}\o\pi_{C_1}$, respectively. 
\item If $C_1,C_2\cu M$ are gate-convex and $C_1\cap C_2\neq\emptyset$, then $C_1\cap C_2$ is gate-convex with gate-projection $\pi_{C_1}\o\pi_{C_2}=\pi_{C_2}\o\pi_{C_1}$. In particular, if $C_2\cu C_1$, we have $\pi_{C_2}=\pi_{C_2}\o\pi_{C_1}$.
\item If $C_1,C_2\cu M$ are gate-convex, we have $\pi_{C_1}\o\pi_{C_2}\o\pi_{C_1}=\pi_{C_1}\o\pi_{C_2}$.
\end{enumerate}
\end{lem}
\begin{proof}
For part~1, see 1.8 and the corollary to 2.5 in \cite{Isbell}. Part~2 follows from Proposition~\ref{retractions new} above and part~3 is an immediate consequence. Part~4 follows from part~3 and the observation that $\pi_{C_1}\o\pi_{C_2}$ is the gate-projection to $\pi_{C_1}(C_2)\cu C_1$.
\end{proof}

Each convex subset $C\cu M$ is also a median subalgebra. In particular, we can consider the collection $\mscr{H}(C)$ of all halfspaces of the median algebra $C$.

\begin{prop}\label{walls in convex}
If $C\cu M$ is gate-convex, there is a one-to-one correspondence
\begin{align*}
\{ \mf{h}\in\mscr{H}(M) \mid \mf{h}\cap C\neq\emptyset, \ \mf{h}^*\cap C\neq\emptyset\} \longleftrightarrow & \mscr{H}(C) \\
\mf{h} \longmapsto & \mf{h}\cap C \\
\pi_C^{-1}(\mf{k}) \longmapsfrom & \mf{k} .
\end{align*}
Moreover, if $\mf{h}\cap C,\mf{k}\cap C\in\mscr{H}(C)$, then $\mf{h}\cu\mf{k}$ if and only if $\mf{h}\cap C\cu\mf{k}\cap C$.
\end{prop}
\begin{proof}
If $\mf{h}\in\mscr{H}(M)$ and $\mf{h}\cap C$, $\mf{h}^*\cap C$ are both nonempty, they are halfspaces of $C$. All halfspaces of $C$ arise this way: given a partition $C=C_1\sqcup C_2$ where $C_i$ are both convex, we obtain a partition $M=\pi_C^{-1}\left(C_1\right)\sqcup\pi_C^{-1}\left(C_2\right)$ and Proposition~\ref{retractions new} ensures that the $\pi_C^{-1}\left(C_i\right)$ are also convex. Finally, part~1 of Lemma~\ref{gates vs inclusions} implies that, if $\mf{h}\cap C\in\mscr{H}(C)$, then $\pi_C(\mf{h})\cu\mf{h}\cap C$; equivalently, $\mf{h}\cu\pi_C^{-1}(\mf{h}\cap C)$. Similarly $\mf{h}^*\cu\pi_C^{-1}(\mf{h}^*\cap C)$, thus $\mf{h}=\pi_C^{-1}(\mf{h}\cap C)$.
\end{proof}

In particular, for all $x,y\in M$, we can identify $\mscr{W}(x|y)\simeq\mscr{W}\left(I(x,y)\right)$ canonically. We will also (slightly improperly) consider the sets $\mscr{H}(C)$ as subsets of $\mscr{H}(M)$ from now on. Given subsets $C_1,C_2\cu M$, we say that $(x_1,x_2)$ is a \emph{pair of gates} for $(C_1,C_2)$ if $x_1$ is a gate for $(x_2,C_1)$ and $x_2$ is a gate for $(x_1,C_2)$.

\begin{lem}\label{gates for pairs of convex sets 1}
If $C_1,C_2\cu M$ are gate-convex, for every $y_1\in C_1$ and ${y_2\in C_2}$ there exists a pair of gates $(x_1,x_2)$ for $(C_1,C_2)$ such that $y_1x_1x_2y_2$ is a discrete geodesic. Moreover, $\mscr{H}(C_1|C_2)=\mscr{H}(x_1|x_2)$. 
\end{lem}
\begin{proof}
Set $x_2:=\pi_{C_2}(y_1)$ and $x_1:=\pi_{C_1}(x_2)$; by part~2 of Lemma~\ref{gates vs inclusions}, we have $\pi_{C_2}(x_1)=\pi_{C_2}\pi_{C_1}(x_2)=x_2$. Hence $(x_1,x_2)$ is a pair of gates and the fact that $y_1x_1x_2y_2$ is a discrete geodesic follows from the gate property. Another consequence of $x_1$ and $x_2$ being gates is that the sets $\mscr{H}(x_1|x_2)$, $\mscr{H}(x_1|C_2)$ and $\mscr{H}(C_1|x_2)$ coincide; this yields the last part of the lemma.
\end{proof}

\begin{lem}\label{rank with a subset}
Let $\mscr{K}\cu\mscr{H}$ be a subset with $\mscr{K}\cap\mscr{H}(x|y)\neq\emptyset$ for every $x,y\in M$. The rank of $M$ coincides with the maximal cardinality of a set of pairwise-transverse halfspaces in $\mscr{K}$.
\end{lem}
\begin{proof}
It suffices to prove that, if $\mf{h}_1,...,\mf{h}_{k-1},\mf{h}\in\mscr{H}$ are pairwise transverse, there exists $\mf{h}'\in\mscr{K}$ such that $\mf{h}_1,...,\mf{h}_{k-1},\mf{h}'$ are pairwise transverse. Pick points $x\in\mf{h}_1^*\cap ... \cap\mf{h}_{k-1}^*\cap\mf{h}^*$, $y\in\mf{h}_1\cap ... \cap\mf{h}_{k-1}\cap\mf{h}^*$, $u\in\mf{h}_1^*\cap ... \cap\mf{h}_{k-1}^*\cap\mf{h}$ and $v\in\mf{h}_1\cap ... \cap\mf{h}_{k-1}\cap\mf{h}$; these exist by Helly's Theorem. The intervals $I:=I(x,y)$ and $J:=I(u,v)$ are disjoint since $I\cu\mf{h}^*$ and $J\cu\mf{h}$; thus there exists $\mf{h}'\in\mscr{H}(I|J)\cap\mscr{K}$, by Lemma~\ref{gates for pairs of convex sets 1}. It is immediate to check that $\mf{h}_1,...,\mf{h}_{k-1},\mf{h}'$ are pairwise transverse.
\end{proof}

A median algebra $M$ is a \emph{topological median algebra} if it is endowed with a Hausdorff topology so that the median map $m$ is continuous. We speak of a \emph{locally convex} median algebra if, in addition, every point has a basis of convex neighbourhoods. A topological median algebra is said to have \emph{compact intervals} if, for every $x,y\in M$, the interval $I(x,y)$ is compact. 

\begin{lem}\label{gate-convexity vs compact intervals}
Let $M$ be a topological median algebra $M$ with compact intervals. A convex subset $C\cu M$ is gate-convex if and only if it is closed.
\end{lem}
\begin{proof}
The fact that gate-convex subsets are closed holds in any topological median algebra. Indeed, if $C$ is gate-convex with projection $\pi$ and $x\not\in C$, the points $x$ and $\pi(x)$ are distinct. Setting $I:=I(x,\pi(x))$, part~1 of Lemma~\ref{gates vs inclusions} gives $\pi_I(C)=I\cap C=\{\pi(x)\}$. Since the median map $m$ is continuous by definition, the projection $\pi_I$ is also continuous and $\pi_I^{-1}(I\setminus\{\pi(x)\})$ is an open neighbourhood of $x$ disjoint from $C$. Hence $C$ is closed.

Now suppose that $M$ has compact intervals and that $C\cu M$ is closed and convex. Given $x\in M$, we consider the family $\mscr{G}:=\{I(x,y)\cap C\mid y\in C\}$; by Helly's Theorem, any two elements of $\mscr{G}$ intersect. Another application of Helly's Theorem shows that $\mscr{G}$ has the finite intersection property. By compactness, the intersection of all elements of $\mscr{G}$ is nonempty. Any point in this intersection is a gate for $(x,C)$; this proves that $C$ is gate-convex.
\end{proof}

\begin{lem}\label{compact median algebras}
Let $M$ be a compact topological median algebra.
\begin{enumerate}
\item Projections to gate-convex sets are continuous.
\item If $C_1,C_2\cu M$ are convex and compact, the convex hull of $C_1\cup C_2$ is compact. 
\end{enumerate}
\end{lem}
\begin{proof}
Suppose that the projection $\pi$ to a gate-convex subset $C\cu M$ is not continuous. There exist $y\in M$ and a net $(y_j)_{j\in J}$ converging to $y$, such that $\pi(y_j)$ does not converge to $\pi(y)$. By compactness, there exists a subnet $(z_k)_{k\in K}$ such that $(\pi(z_k))_{z\in K}$ converges to a point $z\neq\pi(y)$; since $C$ is closed by Lemma~\ref{gate-convexity vs compact intervals}, we have $z\in C$. Thus, $\pi(z_k)=m(z_k, \pi(y), \pi(z_k))$ converges to $m(y,\pi(y),z)=\pi(y)$ for $k\in K$; this implies $z=\pi(y)$, a contradiction. 

For part~2, the map $f\colon M\ra M$ given by $f(x):=m\left(x,\pi_{C_1}(x),\pi_{C_2}(x)\right)$ is continuous by part~1 and Lemma~\ref{gate-convexity vs compact intervals}. The hull of $C_1\cup C_2$ is precisely the fixed-point set of $f$; this easily follows from Proposition~2.3 in \cite{Roller} and the gate-property. We conclude that the hull is closed, hence compact.
\end{proof}

Given a metric space $X$ and $x,y\in X$, we denote by $I(x,y)$ the \emph{interval} between $x$ and $y$, i.e.\ the set of points $z\in X$ such that $d(x,y)=d(x,z)+d(z,y)$. We say that $X$ is a \emph{median space} if, for all $x,y,z\in X$, the intersection $I(x,y)\cap I(y,z)\cap I(z,x)$ consists of a single point, which we denote by $m(x,y,z)$. This defines a median-algebra structure on $X$ with the same notion of interval; in particular, we can define rank, convexity and gate-convexity for subsets of $X$. 

A complete median space is geodesic if and only if it is connected, see Lemma~4.6 in \cite{Bow4}. In this case, the interval $I(x,y)$ is simply the union of all geodesics with endpoints $x$ and $y$. 

If $X$ is a geodesic median space, its rank coincides with the supremum of the topological dimensions of its locally compact subsets --- even when either of the two quantities is infinite; see Theorem~2.2 and Lem\-ma~7.6 in \cite{Bow1} for one inequality and Proposition~5.6 in \cite{Bow4} for the other. We prefer to speak of \emph{rank}, rather than \emph{dimension}, as the metric spaces that we will be interested in could well be disconnected, e.g.\ $0$-skeleta of ${\rm CAT}(0)$ cube complexes. 

\begin{ex}\label{L1 are median}
Let $(\Om,\mscr{B},\mu)$ be a measure space.
\begin{enumerate}
\item The space $L^1(\Om,\mu)$ is median when endowed with the metric induced by its norm. The median map is determined by the property that $m(f,g,h)(x)$ is the middle value of $\{f(x),g(x),h(x)\}$, for almost every $x$ and all $f,g,h\in L^1(\Om,\mu)$.
\item Given any $E\cu\Om$, the collection $\mc{M}_E$ of all $F\cu\Om$ such that $E\triangle F$ is measurable and finite-measure can be given the pseudometric 
\[d(A,B)=\mu(A\triangle B).\] 
This makes sense as $A\triangle B=(A\triangle E)\triangle (B\triangle E)$. Identifying sets at distance zero, the space $\mc{M}_E$ can be isometrically embedded into $L^1(\Om,\mu)$ by mapping $F\mapsto\mathds{1}_{F\triangle E}$ and it inherits a median metric. A point lies in the set $m(A,B,C)$ if and only if it lies in at least two of the sets $A,B,C\cu\Om$;  the interval $I(A,B)$ can be recognised as the collection of sets $Z$ satisfying $A\cap B\cu Z\cu A\cup B$.  
\end{enumerate}
\end{ex}

Let $X$ be a median space throughout the rest of this section. The median map $m$ is $1$-Lipschitz if we endow $X^3$ with the $\ell^1$ metric. If $C\cu X$ is convex and $x\in X$, a point $z\in C$ is a gate for $(x,C)$ if and only if $d(x,C)=d(x,z)$; gate-projections are $1$-Lipschitz. Gate-convex sets are closed and convex; the converse holds in complete median spaces. See \cite{CDH} for further details and examples.

In complete median spaces we can complement Lemma~\ref{gates for pairs of convex sets 1} above.

\begin{lem}\label{gates for pairs of convex sets 2}
If $X$ is complete and $C_1,C_2\cu X$ are closed and convex, the points $z_1\in C_1$, $z_2\in C_2$ form a pair of gates for $(C_1,C_2)$ if and only if ${d(z_1,z_2)=d(C_1,C_2)}$. In particular, disjoint closed convex sets always have positive distance.
\end{lem}
\begin{proof}
If $d(z_1,z_2)=d(C_1,C_2)$, it is immediate that $(z_1,z_2)$ is a pair of gates. Conversely, given a pair of gates $(z_1,z_2)$, we set $I:=I(z_1,z_2)$; if $z_i'\in C_i$, we have $\pi_I(z_i')=z_i$ by part~1 of Lemma~\ref{gates vs inclusions} and the observation that $C_i\cap I=\{z_i\}$. Since $\pi_I$ is $1$-Lipschitz, we have $d(z_1,z_2)\leq d(z_1',z_2')$; hence $d(z_1,z_2)=d(C_1,C_2)$ by the arbitrariness of $z_i'$.
\end{proof}

\begin{lem}\label{locally convex}
If $X$ has finite rank, it is locally convex.
\end{lem}
\begin{proof}
Given $x\in X$, $\eps>0$ and $y,z\in B(x,\eps)$, we have $I(y,z)\cu B(x,2\eps)$. Thus $X$ is ``weakly locally convex'' in the sense of \cite{Bow1} and we conclude by Lemma~7.1 in \cite{Bow1}.
\end{proof}

The class of locally convex median spaces encompasses both finite rank median spaces and infinite dimensional ${\rm CAT}(0)$ cube complexes. Instead, the median space $L^1([0,1])$ is not locally convex; as we shall see in Example~\ref{dense halfspaces in L^1}, the convex hull of any nonempty open subset is the entire $L^1([0,1])$.

\begin{lem}\label{pre-adjacency}
Suppose $X$ is complete and locally convex and let $\{C_i\}_{i\in I}$ be a collection of convex subsets of $X$ with nonempty intersection $K$. The intersection of $\{\overline{C_i}\}_{i\in I}$ is $\overline K$.
\end{lem}
\begin{proof}
We only need to prove that, if $x\in\overline{C_i}$ for all $i\in I$, then $x\in\overline K$. Given $\eps>0$, let $N_{\eps}\cu B(x,\eps)$ be a convex neighbourhood of $x$; denote by $\pi_{\eps}$ the gate-projection to $\overline{N}_{\eps}$. Since $C_i\cap N_{\eps}\neq\emptyset$ for all $i\in I$, part~1 of Lemma~\ref{gates vs inclusions} implies that $\pi_{\eps}(C_i)=C_i\cap\overline{N}_{\eps}$ and
\[\pi_{\eps}(K)\cu\bigcap_i\pi_{\eps}(C_i)\cu\left(\bigcap_i C_i\right)\cap\overline{N}_{\eps}=K\cap\overline{N}_{\eps}.\]
Hence $K$ intersects $B(x,2\eps)$ for all $\eps>0$ and, by the arbitrariness of $\eps$, we conclude that $x\in\overline K$.
\end{proof}

\subsection{SMW's, PMP's and SMH's.}\label{SMW}

Let $X$ be a set. A \emph{wall} is an unordered pair $\{\mf{h},\mf{h}^*\}$ corresponding to any partition $X=\mf{h}\sqcup\mf{h}^*$. The wall \emph{separates} subsets $A,B\cu X$ if $A\cu\mf{h}$ and $B\cu\mf{h}^*$ or vice versa. As in Section~\ref{median algebras}, we use the notation $\mscr{W}(A|B)$ to refer to walls separating $A$ and $B$.

\begin{defn}
We say that the 4-tuple $(X,\mscr{W},\mscr{B},\mu)$ is a \emph{space with measured walls (SMW)} if $\mscr{W}$ is a collection of walls of $X$ and the measure $\mu$, defined on the $\s$-algebra $\mscr{B}\cu 2^{\mscr{W}}$, satisfies $\mu\left(\mscr{W}(x|y)\right)<+\infty$ for all $x,y\in X$.
\end{defn}

If $(X,\mscr{W},\mscr{B},\mu)$ is a space with measured walls, the associated collection of \emph{halfspaces} is the set $\mscr{H}$ of those $\mf{h}\cu X$ such that $\{\mf{h},\mf{h}^*\}\in\mscr{W}$. It is endowed with a two-to-one projection $\pi\colon\mscr{H}\ra\mscr{W}$ given by $\pi(\mf{h})=\{\mf{h},\mf{h}^*\}$. We can define a pseudo-metric on $X$ by setting $\text{pdist}_{\mu}(x,y):=\mu\left(\mscr{W}(x|y)\right)$. When this is a genuine metric, we speak of a \emph{faithful} SMW. 

Let $(X',\mscr{W}',\mscr{B}',\mu')$ be another SMW. Any map $f\colon X\ra X'$ such that $\{f^{-1}(\mf{h}),f^{-1}(\mf{h}^*)\}\in\mscr{W}$ whenever $\{\mf{h},\mf{h}^*\}\in\mscr{W}'$ induces a map $f^*\colon\mscr{W}'\ra\mscr{W}$. If $f^*$ is measurable and $(f^*)_*\mu'=\mu$, we say that $f$ is a \emph{homomorphism} of spaces with measured walls. Note that this definition differs slightly from the one in \cite{CDH}. If $f$ is a homomorphism we have 
\[\mu\left(\mscr{W}(x|y)\right)=\mu'\left((f^*)^{-1}\left(\mscr{W}(x|y)\right)\right)=\mu'\left(\mscr{W}(f(x)|f(y))\right),\]
so the map $f\colon X\ra X'$ preserves the pseudo-metric. In particular, isomorphisms of SMW's are isometries for the pseudo-metrics.

\begin{defn}
A \emph{pointed measured pocset (PMP)} is a 4-tuple $\left(\mscr{P},\mscr{D},\eta,\s\right)$, where $\mscr{P}$ is a pocset, $\mscr{D}$ is a $\s$-algebra of subsets of $\mscr{P}$, the measure $\eta$ is defined on $\mscr{D}$ and $\s\cu\mscr{P}$ is a (not necessarily measurable) ultrafilter.
\end{defn}

In analogy to the terminology of \cite{CDH}, we say that an ultrafilter ${\s'\cu\mscr{P}}$ is \emph{admissible} if $\s'\triangle\s\in\mscr{D}$ and $\eta(\s'\triangle\s)<+\infty$. We will denote by $\mc{M}\left(\mscr{P},\mscr{D},\eta,\s\right)$ (or simply $\mc{M}$) the set of admissible ultrafilters associated to the pointed measured pocset $\left(\mscr{P},\mscr{D},\eta,\s\right)$. As in Example~\ref{L1 are median}, we can equip $\mc{M}$ with the median pseudo-metric $d(\s_1,\s_2):=\eta(\s_1\triangle\s_2)$. We identify admissible ultrafilters at zero distance, so that $\mc{M}$ becomes a median space.

Two PMP's $\left(\mscr{P},\mscr{D},\eta,\s\right)$ and $\left(\mscr{P}',\mscr{D}',\eta',\s'\right)$ are \emph{isomorphic} if there exists an isomorphism of pocsets $f\colon\mscr{P}\ra\mscr{P}'$ such that $f$ and $f^{-1}$ are measurable, $f_*\eta=\eta'$ and $\eta(f^{-1}(\s')\triangle\s)<+\infty$. Any isomorphism of the two PMP's induces an isometry of the corresponding median spaces $\mc{M}$ and $\mc{M}'$.

Given a SMW $(X,\mscr{W},\mscr{B},\mu)$, we always obtain a PMP $(\mscr{H},\pi^*\mscr{B},\pi^*\mu,\s_x)$, where $\s_x\cu\mscr{H}$ is the set of halfspaces containing $x$. Here $\pi^*\mscr{B}$ denotes the $\s$-algebra $\{\pi^{-1}(E)\mid E\in\mscr{B}\}$ and $\pi^*\mu(\pi^{-1}(E))=\mu(E)$. The choice of $x\in X$ does not affect the isomorphism type of the PMP. 

We simply denote by $\mc{M}(X)$ the associated median space of admissible ultrafilters; unlike in \cite{CDH, Chatterji-Drutu}, for us this is a genuine metric space. We have a pseudo-distance-preserving map $X\ra\mc{M}(X)$ given by $y\mapsto\s_y$. If $(X,\mscr{W},\mscr{B},\mu)$ is faithful, this is an isometric embedding. We have recovered the following:

\begin{prop}[\cite{CDH}]\label{walls->median}
Any faithful space with measured walls can be isometrically embedded into a median space.
\end{prop}

The embedding is canonical in that every automorphism of $X$ as SMW extends to an isometry of $\mc{M}(X)$. Note however that the restriction of the metric to $X$ does not have to be median, as $X$ might not be a median subalgebra. The following is a partial converse to the previous proposition. 

\begin{thm}[\cite{CDH}]\label{median->walls}
Let $Y$ be a median space, $\mscr{W}$ its set of convex walls and $\mscr{B}$ be the $\s$-algebra generated by wall-intervals. There exists a measure $\mu$ on $\mscr{W}$ such that $\mu\left(\mscr{W}(x|y)\right)=d(x,y)$ for all $x,y\in Y$. In particular, $(Y,\mscr{W},\mscr{B},\mu)$ is a faithful space with measured walls and we have isometric embeddings
\[Y\hookrightarrow\mc{M}(Y)\hookrightarrow L^1(\mscr{H},\pi^*\mu).\]
\end{thm}

Summing up, we can associate a median space $\mc{M}(X)$ to every faithful SMW $(X,\mscr{W},\mscr{B},\mu)$ and a faithful SMW to every median space. One can wonder whether the compositions 
\[\text{SMW}\leadsto\text{median space}\leadsto\text{SMW},\]
\[\text{median space}\leadsto\text{SMW}\leadsto\text{median space},\]
are the identity. While this has no hope of being true in the first case (we could have taken a set of non-convex walls of a median space), we will show in Corollary~\ref{tangible} that $Y\simeq\mc{M}(Y)$ for all locally convex median spaces $Y$.

We conclude by introducing the following variation on the notion of SMW, which will be more useful to us in the following treatment.

\begin{defn}
The 4-tuple $(X,\mscr{H},\mscr{B},\nu)$ is a \emph{space with measured halfspaces (SMH)} if $\mscr{H}\cu 2^X$ is a collection of subsets of $X$ closed under taking complements, $\mscr{B}\cu 2^{\mscr{H}}$ is a $\s$-algebra and $\nu$ is a measure defined on $\mscr{B}$ satisfying $\nu\left(\mscr{H}(x|y)\right)=\nu\left(\mscr{H}(y|x)\right)<+\infty$ for all $x,y\in X$.
\end{defn}

If $\mf{h}\in\mscr{H}$, we set $\mf{h}^*:=X\setminus\mf{h}$ and, if $E\cu\mscr{H}$, we define $E^*:=\{\mf{h}^*\mid\mf{h}\in E\}$. We borrow the notation $\mscr{H}(A|B)$ and $\s_A$ from Section~\ref{median algebras}. Note that, if $(X,\mscr{W},\mscr{B},\mu)$ is a SMW with associated collection of halfspaces $\mscr{H}$ and projection $\pi\colon\mscr{H}\ra\mscr{W}$, the 4-tuple $(X,\mscr{H},\pi^*\mscr{B},\pi^*\mu)$ is \emph{not} a space with measured halfspaces; indeed, $\mscr{H}(x|y)\not\in\pi^*\mscr{B}$ if $x\neq y$. 

A pseudo-metric on $X$ and a notion of homomorphism of SMH's can be defined exactly as we did for SMW's. Given a space with measured halfspaces $(X,\mscr{H},\mscr{B},\nu)$ and a point $x\in X$, we obtain a PMP $(\mscr{H},\mscr{B},\nu,\s_x)$; we write $\mc{M}(X)$ instead of $\mc{M}(\mscr{H},\mscr{B},\nu,\s_x)$. The discussion in Section~5 of \cite{CDH} works identically if we replace the symbol $\mscr{W}$ with $\mscr{H}$ everywhere. In particular, we have:

\begin{thm}\label{median->halfspaces}
Let $Y$ be a median space, $\mscr{H}$ the set of convex halfspaces and $\mscr{B}$ be the $\s$-algebra generated by halfspace-intervals. There exists a measure $\nu$ on $\mscr{H}$ such that $\nu\left(\mscr{H}(x|y)\right)=d(x,y)$ for all $x,y\in Y$. In particular, $(Y,\mscr{H},\mscr{B},\nu)$ is a faithful space with measured halfspaces and we have isometric embeddings
\[Y\hookrightarrow\mc{M}(Y)\hookrightarrow L^1(\mscr{H},\nu).\]
\end{thm}

Note that $*\colon\mscr{H}\ra\mscr{H}$ is a measure-preserving involution. We will always denote the $\s$-algebras in Theorems~\ref{median->walls} and~\ref{median->halfspaces} by the same symbol as there is no chance of confusion.

\begin{rmk}\label{bad ultrafilters on R}
Embedding $Y\hookrightarrow\mc{M}(Y)$ as in Theorem~\ref{median->halfspaces}, each point $y\in Y$ is represented by all measurable ultrafilters $\s\cu\mscr{H}$ that have null symmetric difference with $\s_y$. Some of these ultrafilters can be ``bad'': the intersection of all halfspaces in $\s$ can be empty, rather than $\{y\}$. 

As an example, consider $Y=\R$ with its standard metric and $y=0$. The natural ultrafilter $\s_0$ representing $0$ consists of the halfspaces $(-\infty,a)$ for $a>0$, $(-\infty, a]$ for $a\geq 0$, $(a,+\infty)$ for $a<0$ and $[a,+\infty)$ for $a\leq 0$. However, $\overline\s_0:=\left(\s_0\setminus\{(-\infty, 0]\}\right)\cup\{(0,+\infty)\}$ also is an admissible ultrafilter representing $0$. Here, measurability and nullity of $\s_0\triangle\overline\s_0$ follow from the observation that $\{(-\infty, 0]\}$ coincides with the intersection of the countable family of halfspace-intervals $\mscr{H}(\tfrac{1}{n}|0)$, $n\geq 1$.
\end{rmk}

We remark that the measures in Theorems~\ref{median->walls} and~\ref{median->halfspaces} can be extended to \emph{complete} $\s$-algebras $\mscr{B}_0\supseteq\mscr{B}$, i.e.\ with the property that any subset of a null set is measurable; see e.g.\ Theorem~1.36 in \cite{Rudin}.

\subsection{Intervals and halfspaces.}

Let $X$ be a median space throughout this section. It is not hard to use the arguments of \cite{Bow2} to prove the following generalisation of Theorem~1.14 in \cite{BCGNW}; still, we provide a proof below for the convenience of the reader.

\begin{prop}\label{intervals are Euclidean}
Let $X$ be complete of rank $r<+\infty$. For every $x,y\in X$, there exists an isometric embedding $I(x,y)\hookrightarrow\R^r$, where $\R^r$ is endowed with the $\ell^1$ metric.
\end{prop}

In particular, we obtain the following useful fact (for connected median spaces, also see Corollary~1.3 in \cite{Bow2}).

\begin{cor}\label{intervals are compact new}
In a complete, finite rank median space intervals are compact.
\end{cor}

To prove Proposition~\ref{intervals are Euclidean}, observe that, if $M$ is a median algebra and $\s_1,\s_2\cu\mscr{H}(M)$ are ultrafilters, antichains in the poset $\s_1\setminus\s_2$ correspond to sets of pairwise-transverse halfspaces and thus have cardinality bounded above by $\text{rank}(M)$. Hence, Dilworth's Theorem \cite{Dilworth} yields the following:

\begin{lem}\label{Dilworth for differences}
Let $M$ be a median algebra with $\text{rank}(M)=r<+\infty$ and let $\s_1,\s_2$ be ultrafilters on $\mscr{H}$. We can decompose $\s_1\setminus\s_2=\mc{C}_1\sqcup ... \sqcup\mc{C}_k$, where $k\leq r$ and each $\mc{C}_i$ is nonempty and totally ordered by inclusion.
\end{lem}

Note that, in general, we have no guarantee that the chains provided by the previous lemma are measurable.

\begin{proof}[Proof of Proposition~\ref{intervals are Euclidean}.]
Assume that $X=I(x,y)$ for simplicity. We will produce maps $f_1,...,f_r\colon X\ra [0,d(x,y)]$ such that for every $u,v\in X$ we have $d(u,v)=\sum\left|f_i(u)-f_i(v)\right|$. 

We first do this under the assumption that $X$ be finite. By Lemma~\ref{Dilworth for differences}, we can decompose $\mscr{H}(x|y)=\mc{C}_1\sqcup ... \sqcup\mc{C}_r$, where each $\mc{C}_i$ is a finite set that is totally ordered by inclusion. Let $\nu$ be the measure on $\mscr{H}$ that is provided by Theorem~\ref{median->halfspaces}. Since $X$ is finite, the singletons of $\mscr{H}$ are halfspace-intervals (see e.g.\ Section~3 in \cite{Bow4}), hence measurable; thus, each $\mc{C}_i$ is measurable. For $z\in X$ and $1\leq i\leq r$, we set $f_i(z):=\nu\left(\mc{C}_i\cap\mscr{H}(x|z)\right)$. It is straightforward to check that these define the required embedding.

In the general case, let $\mf{M}$ be the set of finite subalgebras of $X$ containing $\{x,y\}$. Every finite subset of $X$ is contained in an element of $\mf{M}$ by Lemma~4.2 in \cite{Bow1}; in particular, $\left(\mf{M},\cu\right)$ is a directed set. Every $M\in\mf{M}$ is a finite interval with endpoints $x,y$, so the previous discussion yields maps $f_1^M,...,f_r^M\colon M\ra [0,d(x,y)]$ defining an isometric embedding $M\hookrightarrow\R^r$. We can extend each $f_i^M$ to a function $\wt{f}_i^M\colon X\ra [0,d(x,y)]$ that takes $X\setminus M$ to zero. This defines nets $P_i\colon\mf{M}\longrightarrow [0,d(x,y)]^X$. The space $[0,d(x,y)]^X$ is compact by Tychonoff's Theorem, thus a subnet of $P_i$ converges. Its limit is a function $f_i\colon X\ra [0,d(x,y)]$ and it is immediate to check that $f_1,...,f_r$ yield the required embedding $X\hookrightarrow\R^r$.
\end{proof}

We now proceed to examine various properties of the halfspaces of $X$.

\begin{prop}\label{boundaries of halfspaces}
If $X$ has finite rank and $\mf{h}\in\mscr{H}$, either $\partial\mf{h}:=\overline{\mf{h}}\cap\overline{\mf{h}^*}$ is empty or it is a closed, convex subset with $\text{rank}(\partial\mf{h})\leq\text{rank}(X)-1$.
\end{prop}
\begin{proof}
Follows from Lemma~\ref{locally convex} above and Lemma~7.5 in \cite{Bow1}. 
\end{proof}

We remark that, in a median space, closures of convex sets are convex, but interiors of convex sets need not be; in particular, the closure of a halfspace needs not be a halfspace. For instance, consider a real tree $T$, a branch point $x\in T$ and a connected component $\mf{h}$ of $T\setminus\{x\}$. This is a halfspace of $T$, as both $\mf{h}$ and $\mf{h}^*$ are convex. The interior of $\mf{h}^*$, however, is not convex, as it coincides with $\mf{h}^*\setminus\{x\}$. In particular, $\overline{\mf{h}}=\mf{h}\cup\{x\}$ is not a halfspace. 

Nevertheless, we have the following:

\begin{cor}\label{open/closed}
In a complete, finite rank median space, each halfspace is either open or closed (possibly both).
\end{cor}
\begin{proof}
We proceed by induction on rank$(X)$; if the rank is zero, $\mscr{H}=\emptyset$ and there is nothing to prove. Now assume the result for all median spaces of rank at most $\text{rank}(X)-1$ and suppose $\mf{h}\in\mscr{H}(X)$ is neither open nor closed. Then $\partial\mf{h}$ is nonempty and we have a partition $\partial\mf{h}=\left(\partial\mf{h}\cap\mf{h}\right)\sqcup\left(\partial\mf{h}\cap\mf{h}^*\right)$. By Helly's Theorem, the convex set $\partial\mf{h}\cap\mf{h}=\overline{\mf{h}}\cap\overline{\mf{h}^*}\cap\mf{h}$ is nonempty; the same argument yields $\partial\mf{h}\cap\mf{h}^*\neq\emptyset$. The above partition of $\partial\mf{h}$ must then arise from a halfspace of $\partial\mf{h}$ and the inductive hypothesis guarantees that $\partial\mf{h}\cap\mf{h}$ is either open or closed. Note that $\mf{h}=\pi_{\partial\mf{h}}^{-1}\left(\partial\mf{h}\cap\mf{h}\right)$, by Proposition~\ref{walls in convex}. Since $\pi_{\partial\mf{h}}$ is continuous, $\mf{h}$ is either open or closed, a contradiction.
\end{proof}

The situation can be completely different in infinite rank median spaces:

\begin{ex}\label{dense halfspaces in L^1}
The space $X=L^1([0,1])$ is complete, median and all its halfspaces are dense. In order to see the latter, let us write $B_R$ for the $R$-ball around the origin and let us denote characteristic functions by $\chi_{\cdot}$. Given a function $f\in L^1([0,1])$ with $\|f\|_1<2R$, there exists a measurable partition $[0,1]=P\sqcup Q$ such that $\|f\cdot\chi_P\|_1<R$ and $\|f\cdot\chi_Q\|_1<R$. By part~1 of Example~\ref{L1 are median}, the interval between $f\cdot\chi_P$ and $f\cdot\chi_Q$ contains $f$. This shows that $B_{2R}$ is contained in the convex hull of $B_R$ for every $R>0$; in particular, the hull of $B_R$ is the entire $X$. Since any nonempty open subset of $X$ can be translated so that it contains a ball around the origin, we conclude that the hull of any nonempty open set is the entire $X$. Thus, every proper convex subset of $X$ must have empty interior and all halfspaces are dense.
\end{ex}

\noindent
{\bf Example~\ref{dense halfspaces in L^1}'.} Let us consider the compact topological median algebra $M=\{0,1\}^{\N}$ with the product topology; we identify $M$ with the power set $2^{\N}$ of $\N$. There is a one-to-one correspondence between walls of $M$ and ultrafilters\footnote{Here we consider the classical set-theoretical notion of ultrafilter, not the one introduced in Section~\ref{median algebras}. See e.g.\ Definition~10.12 in \cite{DK}.} $\mc{U}\cu 2^{\N}$. Indeed, observe that, given $A,B\cu\N$, we have $I(A,B)=\{C\cu\N\mid A\cap B\cu C\cu A\cup B\}$. Since $M=I(\emptyset,\N)$, every wall of $M$ has a side $\mf{h}\cu 2^{\N}$ containing $\N$ and a side $\mf{h}^*$ containing $\emptyset$. Since $A\cap B\in I(A,B)$, the collection $\mf{h}\cu 2^{\N}$ is closed under taking intersections. If $A\cu B$ and $A\in\mf{h}$, we have $B\in I(A,\N)\cu\mf{h}$. Finally, for every $A\cu\N$, the collection $\mf{h}$ contains exactly one among $A$ and $\N\setminus A$, as $M=I(A,\N\setminus A)$. This shows that $\mf{h}\cu 2^{\N}$ is an ultrafilter. Conversely, it is easy to check that every ultrafilter $\mc{U}\cu 2^{\N}$ is a halfspace of $M$.

Now, $M$ has some obvious walls coming from its product structure. The associated halfspaces can be described explicitly by setting one coordinate of $\{0,1\}^{\N}$ to $0$ or $1$; note that, in a finite product $\{0,1\}^n$, all halfspaces would be of this form. In terms of the correspondence established above, these walls are exactly \emph{principal} ultrafilters (see e.g.\ Definition~10.15 in \cite{DK}). However, it is a well-known consequence of the axiom of choice that there exist also \emph{non-principal} ultrafilters. These will correspond to additional walls of $M$, which --- unlike the previous ones --- yield dense halfspaces of $M$.

Endowing the product $X=\prod \{0,\tfrac{1}{n^2}\}$ with its $\ell^1$ metric, we obtain a compact median metric space. As a median algebra, this is isomorphic to $M$ above, which yields a natural correspondence between walls of $X$ and walls of $M$. The space $X$ is totally disconnected, but the walls of the \emph{geodesic} median space $Y=\prod [0,\tfrac{1}{n^2}]$ are not much better behaved: by Lemma~6.5 in \cite{Bow1}, every halfspace of $X$ is of the form $\mf{h}\cap Y$ for some $\mf{h}\in\mscr{H}(Y)$. We remark that --- unlike $L^1([0,1])$ --- $M$, $X$ and $Y$ are also locally convex.
\medskip

\begin{figure}
\centering
\includegraphics[width=2.4in]{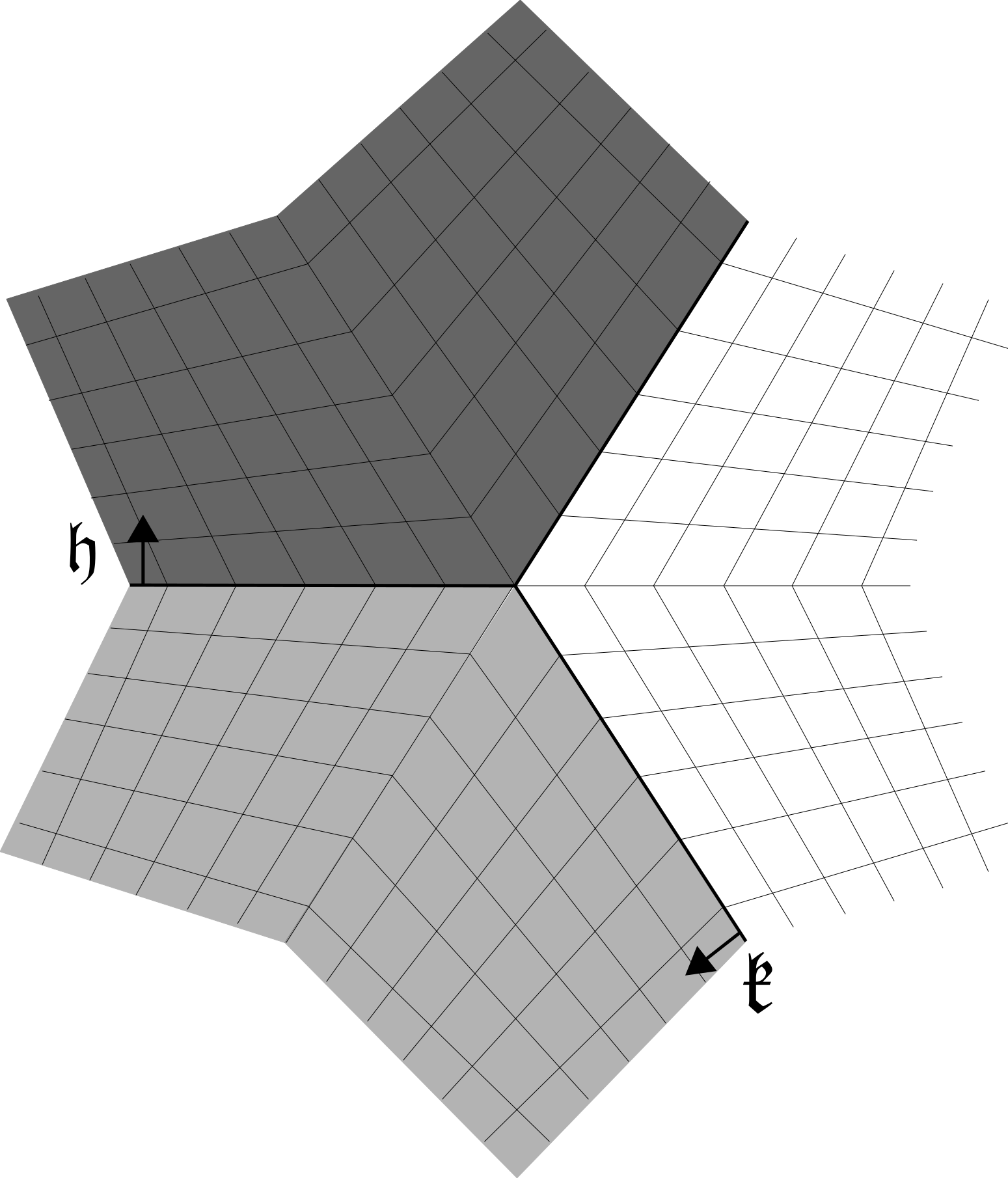}
\caption{}
\label{hexagonal CSC}
\end{figure}
Even in finite rank median spaces, walls can display more complicated behaviours than hyperplanes in ${\rm CAT}(0)$ cube complexes. Consider for instance the rank-two median space pictured in Figure~\ref{hexagonal CSC}; it is obtained by glueing together three halfplanes, each endowed with the $\ell^1$ metric. The pictured halfspaces ($\mf{h}$ is open and $\mf{k}$ is closed) satisfy $\mf{h}\subsetneq\mf{k}$, but $d(\mf{h},\mf{k}^*)=0$; indeed, $\mf{h}$ and $\mf{k}$ share a portion of their frontier isometric to a ray. 

Another pathology appears in the space $I$ in Figure~\ref{butterflies} below, which we view as a subset of $\R^2$ with the restriction of the $\ell^1$ metric. The halfspaces $\mf{h}$, $\mf{k}$ are open and $\mf{h}\subsetneq\mf{k}$, but $\overline{\mf{h}}\not\cu\mf{k}$.

These issues can easily be circumvented, at least in finite rank spaces; this is the content of Proposition~\ref{trivial chains of halfspaces} below.

\begin{lem}\label{strictly increasing distance to halfspaces}
Let $X=I(x,y)$ be complete and finite rank. If, for $\mf{h},\mf{k}\in\mscr{H}$, we have $y\in\mf{h}\cu\mf{k}$ and $d(x,\mf{h})=d(x,\mf{k})$, then $\overline{\mf{h}}=\overline{\mf{k}}$.  
\end{lem}
\begin{proof}
Observe that the gate-projections of $x$ to $\overline{\mf{k}}$ and $\overline{\mf{h}}$ coincide by part~3 of Lemma~\ref{gates vs inclusions}; we denote them by $z$. If $w\in\overline{\mf{k}}$, the sequence $xzwy$ is a discrete geodesic; thus, $w\in\overline{\mf{h}}$. This proves that $\overline{\mf{k}}\cu\overline{\mf{h}}$, while the other inclusion is obvious.
\end{proof}

The hypotheses of Lemma~\ref{strictly increasing distance to halfspaces} do not imply that $\mf{h}=\mf{k}$ or $\overline{\mf{h}}=\mf{k}$; see for instance $\mf{h}\cu\mf{k}$ in Figure~\ref{butterflies}.
\begin{figure}
\centering
\includegraphics[width=4.5in]{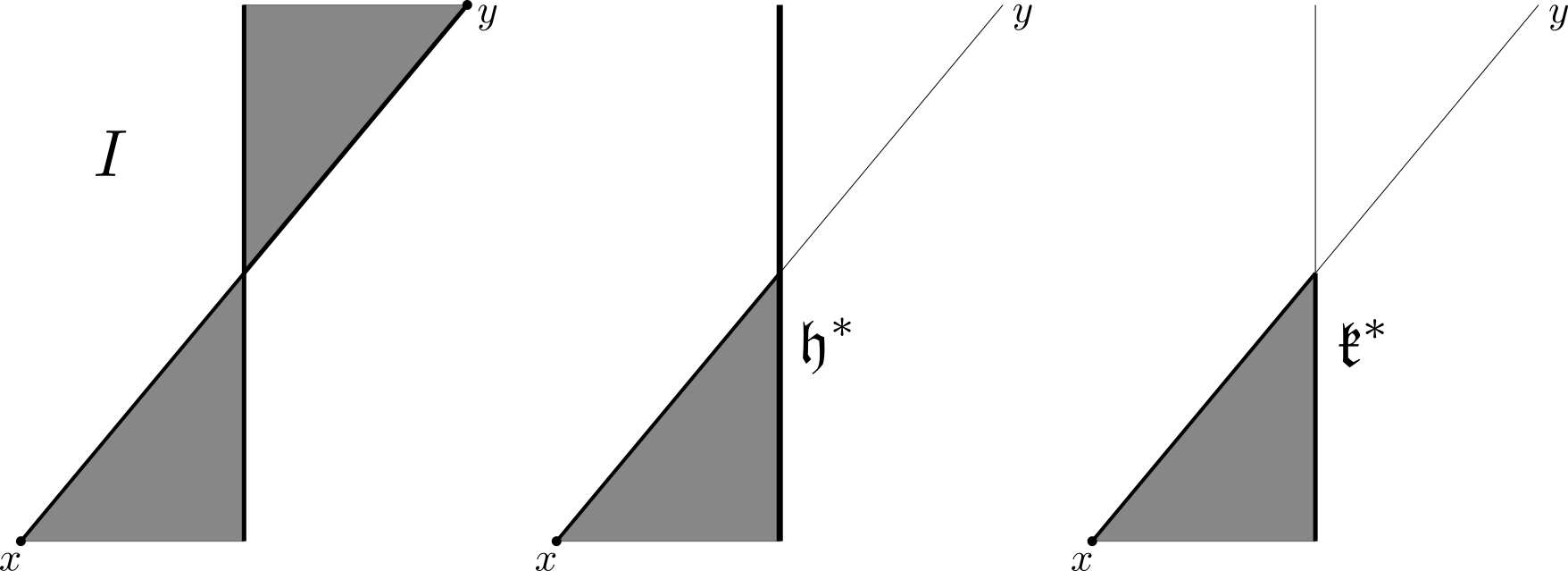}
\caption{}
\label{butterflies}
\end{figure}

\begin{prop}\label{trivial chains of halfspaces}
Let $X$ be complete of rank $r<+\infty$ and let $\mf{h}_1\supsetneq ... \supsetneq\mf{h}_k$ be a chain of halfspaces. 
\begin{enumerate}
\item If $d(\mf{h}_1^*,\mf{h}_k)=0$ and each $\mf{h}_i$ is open, then $k\leq r$. 
\item In general, if $d(\mf{h}_1^*,\mf{h}_k)=0$ we have $k\leq 2r$.
\item If there exists $x\in\mf{h}_1^*$ such that $d(x,\mf{h}_1)=d(x,\mf{h}_k)$, then $k\leq r+1$.
\end{enumerate}
\end{prop}
\begin{proof}
If $d(\mf{h}_1^*,\mf{h}_k)=0$, no $\mf{h}_i$ can simultaneously be open and closed or we would have $\overline{\mf{h}_1^*}\cap\overline{\mf{h}_k}\cu\overline{\mf{h}_i^*}\cap\overline{\mf{h}_i}=\mf{h}_i^*\cap\mf{h}_i=\emptyset$ and $d(\mf{h}_1^*,\mf{h}_k)>0$ by Lemma~\ref{gates for pairs of convex sets 2}.

We prove part~1 by induction on $r$; the case $r=0$ is trivial. If $r\geq 1$, observe that $C:=\partial\mf{h}_k$ is closed, convex and nonempty, since $\mf{h}_k$ is not closed. If $i\leq k-1$, we have $\mf{h}_i\supseteq\mf{h}_k$ and $\mf{h}_i\cap\mf{h}_k^*\neq\emptyset$, hence $\mf{h}_i\cap 
 C\neq\emptyset$ by Helly's Theorem. Similarly $\mf{h}_i^*\cap C\neq\emptyset$, since $\mf{h}_i^*\cu\mf{h}_k^*$ and since by assumption we have $\mf{h}_i^*\cap\overline{\mf{h}_k}\supseteq\mf{h}_1^*\cap\overline{\mf{h}_k}=\overline{\mf{h}_1^*}\cap\overline{\mf{h}_k}\neq\emptyset$.
 
Proposition~\ref{walls in convex} implies that $\mf{h}_i\cap C$ are a chain of distinct halfspaces of $C$, for $i\leq k-1$ and, by Proposition~\ref{boundaries of halfspaces}, the rank of $C$ is at most ${r-1}$. By Helly's Theorem and Lemma~\ref{gates for pairs of convex sets 2}, we have $\overline{\mf{h}_1^*}\cap\overline{\mf{h}_{k-1}}\cap C\neq\emptyset$ and, by Lemma~\ref{pre-adjacency}, the sets $\overline{\mf{h}_1^*\cap C}$ and $\overline{\mf{h}_{k-1}\cap C}$ intersect. We conclude by applying the inductive hypothesis.

Part~2 is immediate from part~1 and Corollary~\ref{open/closed}, by splitting the chain into a subchain of open halfspaces and a subchain of halfspaces with open complement. To prove part~3, pick a point $y\in\mf{h}_k$ and set $I:=I(x,y)$. Note that $d(x,\mf{h}_i)=d(x,\mf{h}_i\cap I)$, since $\pi_I$ is $1$-Lipschitz, fixes $x$ and maps $\mf{h}_i$ onto $\mf{h}_i\cap I$ by Lemma~\ref{gates vs inclusions}. An application of Lemma~\ref{strictly increasing distance to halfspaces} yields $\overline{\mf{h}_1\cap I}=\overline{\mf{h}_k\cap I}$; by Proposition~\ref{walls in convex}, the $\mf{h}_i\cap I$ are pairwise distinct so $\mf{h}_i\cap I$ cannot be closed if $i\geq 2$. Corollary~\ref{open/closed} and Proposition~\ref{walls in convex} imply that $\mf{h}_i$ is open for $i\geq 2$. Moreover, $\overline{\mf{h}_2^*\cap I}\cap\overline{\mf{h}_k\cap I}=\overline{\mf{h}_2^*\cap I}\cap\overline{\mf{h}_2\cap I}\neq\emptyset$, so $k-1\leq r$ by part~1.
\end{proof}

Each of the bounds in Proposition~\ref{trivial chains of halfspaces} is sharp. An example with $r=2$ is given by the median space in Figure~\ref{butterflies}. For part~1 one can consider the chain $\mf{h}\cu\mf{k}$ and $\mf{h}\cu\mf{k}\cu\overline{\mf{k}}$ for part~3; for part~2 one needs to add $(\tau\mf{h})^*$, where $\tau$ is the natural involution of $I$.

\begin{cor}\label{countable cofinal subsets}
Let $X$ be complete and finite rank. Every totally ordered subset $\mc{C}\cu\mscr{H}$ has a countable subset $\mc{C}_0\cu\mc{C}$ that is cofinal in $\mc{C}$.
\end{cor}
\begin{proof}
Pick a point $x\in X$ and define $\delta_x\colon\mc{C}\ra\R$ by $\delta_x(\mf{h}):=d(x,\mf{h}^*)$ if $x\in\mf{h}$ and $\delta_x(\mf{h}):=-d(x,\mf{h})$ if $x\in\mf{h}^*$. The map $\delta_x$ is monotone and, by part~3 of Proposition~\ref{trivial chains of halfspaces}, it has finite fibres. The image of $\delta_x$ is a separable metric space; let $A\cu\delta_x(\mc{C})$ be a countable dense subset. If $\delta_x(\mc{C})$ has a maximum or a minimum, add them to $A$ and set $\mc{C}_0:=\delta_x^{-1}(A)$. It is immediate to check that this is cofinal in $\mc{C}$ (upwards and downwards).
\end{proof}

\section{Measure theory on the pocset of halfspaces.}\label{measure theory}

\subsection{A finer $\s$-algebra on $\mscr{H}$.}\label{finer s-algebra}

Let $X$ be a median space. The $\s$-algebras introduced in Theorems~\ref{median->walls} and~\ref{median->halfspaces} are often too restrictive to work comfortably with. As an example, the ultrafilters $\s_x$, $x\in X$, need not be measurable, for instance when working with non-separable median spaces. The latter might look like pathological examples, but we remark that, on the contrary, an interesting class of finite rank median spaces arises from asymptotic cones of coarse median groups and these are rarely separable.

Since the measures constructed in Theorems~\ref{median->walls} and~\ref{median->halfspaces} originate from Carath\'eodory's construction, one might wish to consider instead the $\s$-algebras of additive sets for the outer measures $\mu^*$ and $\nu^*$. Unfortunately, even in simple examples, one might obtain measure spaces that are not semifinite. We choose a different path.

We say that a subset $E\cu\mscr{H}$ is \emph{morally measurable} if $E\cap\mscr{H}(x|y)$ lies in $\mscr{B}_0$ for all $x,y\in X$; here $\mscr{B}_0$ is the completion of $\mscr{B}$, see the end of Section~\ref{SMW}. Morally measurable sets form a $\s$-algebra $\wh{\mscr{B}}\supseteq\mscr{B}$. For every $z\in X$, the ultrafilter $\s_z\cu\mscr{H}$ is morally measurable since $\s_z\cap\mscr{H}(x|y)=\mscr{H}(x|m)$, where $m=m(x,y,z)$. 

We say that a morally measurable subset $E\cu\mscr{H}$ is \emph{morally null} if $\nu\left(E\cap\mscr{H}(x|y)\right)=0$ for all $x,y\in X$. In particular, subsets of morally null sets are morally null. Given a morally measurable set $E$, we define
\[\wh{\nu}(E):=\sup\left\{\sum_{i\in I}\nu\left(E\cap\mscr{H}(x_i|y_i)\right)~\middle | ~ \bigsqcup_{i\in I}\mscr{H}(x_i|y_i)\cu\mscr{H}\right\}.\]
We allow $I$ to be of any cardinality, although restricting to finite or countable index sets would not affect the value of $\wh{\nu}$. A morally measurable set $E$ is morally null if and only if $\wh{\nu}(E)=0$.

We recall that a measure space $(\Om,\mu)$ is \emph{semifinite} if every measurable subset $E\cu\Om$ with $\mu(E)=+\infty$ contains a measurable subset $F\cu E$ with $0<\mu(F)<+\infty$.

\begin{prop}\label{the semifinite measure}
The triple $\left(\mscr{H},\wh{\mscr{B}},\wh{\nu}\right)$ is a semifinite measure space.
\end{prop}

We prove Proposition~\ref{the semifinite measure} below; the nontrivial part is showing that $\wh{\nu}$ is a measure. Since differences of halfspace-intervals are finite disjoint unions of halfspace-intervals (see \cite{CDH}), the following is an easy observation.

\begin{lem}\label{countable disjoint unions of intervals}
Any finite (countable) union of halfspace-intervals is a finite (countable) \emph{disjoint} union of halfspace-intervals.
\end{lem}

\begin{proof}[Proof of Proposition~\ref{the semifinite measure}]
It suffices to prove that $\wh{\nu}$ is additive and $\s$-subad\-di\-tive. Let $\{E_n\}_{n\geq 0}$ be a collection of pairwise-disjoint, morally measurable sets with union $E$. Given pairwise-disjoint $\mscr{H}(x_1|y_1),...,\mscr{H}(x_k|y_k)$, we have
\[\sum_{i=1}^k\nu\left(E\cap\mscr{H}(x_i|y_i)\right)=\sum_{n\geq 0}\sum_{i=1}^k\nu\left(E_n\cap\mscr{H}(x_i|y_i)\right)\leq\sum_{n\geq 0}\wh{\nu}(E_n)\]
and this proves the inequality $\wh{\nu}(E)\leq\sum\wh{\nu}(E_n)$. Now, if $E, F$ are disjoint morally measurable sets and $\eps>0$, we can find finitely many points ${x_i,y_i\in X}$ such that $\wh{\nu}(E)-\eps\leq\sum\nu\left(E\cap\mscr{H}(x_i|y_i)\right)$ and, similarly, points $u_j,v_j\in X$ satisfying an analogous inequality for $F$. By Lemma~\ref{countable disjoint unions of intervals}, the union of all $\mscr{H}(x_i|y_i)$ and $\mscr{H}(u_j|v_j)$ can be decomposed as a finite disjoint union of halfspace-intervals $\mscr{H}(w_k|z_k)$. Thus,
\begin{align*}
\wh{\nu}(E)+\wh{\nu}(F)-2\eps&\leq\sum_{i}\nu\left(E\cap\mscr{H}(x_i|y_i)\right)+\sum_{j}\nu\left(F\cap\mscr{H}(u_j|v_j)\right) \\
&\leq\sum_{k}\nu\left((E\sqcup F)\cap\mscr{H}(w_k|z_k)\right)\leq\wh{\nu}(E\sqcup F)
\end{align*}
and, by the arbitrariness of $\eps$, we obtain $\wh{\nu}(E)+\wh{\nu}(F)\leq\wh{\nu}(E\sqcup F)$. We have already shown subadditivity, so $\wh{\nu}(E)+\wh{\nu}(F)=\wh{\nu}(E\sqcup F)$.
\end{proof}

\begin{lem}[Properties of $\wh{\nu}$]\label{properties of nu hat}
\begin{enumerate}
\item For every $E\in\wh{\mscr{B}}$ there exist pairwise-disjoint $\mscr{H}(x_n|y_n)$, $n\geq 0$, such that $\wh{\nu}(E)=\sum\nu\left(E\cap\mscr{H}(x_n|y_n)\right)$.
\item $\wh{\nu}(E)\leq\nu(E)$ for all $E\in\mscr{B}$; in particular, $\wh{\nu}\ll\nu$.
\item $\wh{\nu}(E)=\nu(E)$ if $E\in\mscr{B}$ and $\nu(E)<+\infty$; in particular, if $E$ is morally null, $\nu(E)$ is either $0$ or $+\infty$.
\end{enumerate}
\end{lem}
\begin{proof}
Part~1 follows from Lemma~\ref{countable disjoint unions of intervals} and part~2 is direct from definitions. The ring used to define the measure $\nu$ consists of unions of halfspace-intervals (see \cite{CDH}); thus, if $\nu(E)<+\infty$, the set $E$ is contained in a countable union of halfspace-intervals, hence in a disjoint one, by Lemma~\ref{countable disjoint unions of intervals}. Now part~3 is straightforward.
\end{proof}

\subsection{Properties of the $\s$-algebra $\wh{\mscr{B}}$.}\label{properties of the s-algebra}

Let $X$ be a complete median space throughout this section.

\begin{lem}\label{atoms}
Singletons in $\mscr{H}$ are morally measurable and the following are equivalent for $\mf{h}\in\mscr{H}$:
\begin{enumerate}
\item $\mf{h}$ is an atom for $\wh{\nu}$;
\item $\mf{h}$ is clopen;
\item $d(\mf{h},\mf{h}^*)>0$.
\end{enumerate} 
\end{lem}
\begin{proof}
If one side of the wall $\{\mf{h},\mf{h}^*\}$ is not closed, say $\mf{h}$, we can find points $x_n\in\mf{h}$ converging to some $y\in\mf{h}^*$. Thus, $\{\mf{h}\}$ lies in the intersection of the sets $\mscr{H}(y|x_n)$ and it is morally null (hence morally measurable, since the $\s$-algebra $\mscr{B}_0$ is complete). Otherwise, $\mf{h}$ is clopen; by Lemma~\ref{gates for pairs of convex sets 2}, this is equivalent to $d(\mf{h},\mf{h}^*)>0$. Lemma~\ref{gates for pairs of convex sets 1} provides a pair $(x,y)$ of gates for $(\mf{h},\mf{h}^*)$, hence the set $\{\mf{h}\}=\mscr{H}(\mf{h}^*|\mf{h})=\mscr{H}(y|x)$ lies in $\mscr{B}$ and has positive measure. Conversely, if $\mf{h}$ is an atom, it is easy to see that $d(u,v)\geq\wh{\nu}(\{\mf{h}\})$ for all $u\in\mf{h}$, $v\in\mf{h}^*$.
\end{proof}

\begin{lem}\label{connected vs atoms}
A complete median space $X$ is connected if and only if the measure space $\left(\mscr{H}(X),\wh{\nu}\right)$ has no atoms. 
\end{lem}
\begin{proof}
If $X$ is connected, $\wh{\nu}$ has no atoms by Lemma~\ref{atoms}. Conversely, if $\wh{\nu}$ has no atoms, we prove that, for every $x,y\in X$, there exists $z\in I(x,y)$ such that $d(x,z)=d(z,y)$. This implies that $X$ is geodesic (see e.g.\ Remark~1.4(1) in Chapter~I.1 of \cite{BH}). Let $\mscr{F}_x\cu I(x,y)$ be the subset of points $z$ such that $d(x,z)\leq d(z,y)$; we endow it with a structure of poset by declaring that $z_1\preceq z_2$ whenever $z_2\in I(z_1,y)$. Chains in $\mscr{F}_x$ correspond to Cauchy nets in $I(x,y)$ and these converge; thus, $\mscr{F}_x$ is inductive and Zorn's Lemma yields a maximal element $\overline z$. Exchanging the roles of $x$ and $y$, the same construction provides a point $\overline w$. Suppose for the sake of contradiction that we have $d(x,\overline z)<\frac{1}{2}d(x,y)<d(x,\overline w)$. By maximality of $\overline z$ and $\overline w$, the interval $I(\overline z,\overline w)$ consists of the sole points $\overline z$ and $\overline w$. By Proposition~\ref{walls in convex}, we conclude that $\mscr{H}(\overline z|\overline w)$ consists of a single halfspace, hence an atom, contradiction.
\end{proof}

Given a point $x\in X$ and a convex subset $C\cu X$ we define their \emph{adjacencies}:
\begin{align*}
\text{Adj}_x&:=\{\mf{h}\in\mscr{H}\mid x\not\in\mf{h}, \ x\in\overline{\mf{h}} \}, \\
\text{Adj}(C)&:=\{\mf{h}\in\mscr{H}\mid \mf{h}\cap C=\emptyset, \ \mf{h}\cap\overline C\neq\emptyset \}.
\end{align*}
Note that in general $\text{Adj}_x\neq\text{Adj}(\{x\})=\emptyset$.

\begin{lem}\label{adjacency}
If $X$ is locally convex, adjacencies are morally null. In particular, $\s_C$ and $\mscr{H}(C)$ are morally measurable for every convex subset $C\cu X$.
\end{lem}
\begin{proof}
Given $x,y\in X$, let $K$ be the intersection of all halfspaces in $\text{Adj}_x\cap\s_y$; by Lemma~\ref{pre-adjacency}, we have $x\in\overline K$ and we can find points $x_n\in K$ converging to $x$. Thus $\text{Adj}_x\cap\s_y$ is contained in the intersection of the sets $\mscr{H}(x|x_n)$ and it is morally null. By the arbitrariness of $y$, we conclude that $\text{Adj}_x$ is morally null. For every convex subset $C\cu X$ and points $u,v\in X$, we have
\[\text{Adj}(C)^*\cap\mscr{H}(u|v)\cu\text{Adj}(C)^*\setminus\s_u\cu\mscr{H}(\overline u|C),\]
where $\overline u$ is the gate-projection of $u$ to $\overline C$; the last inclusion follows from part~1 of Lemma~\ref{gates vs inclusions}. The set $\mscr{H}(\overline u|C)\cu\text{Adj}_{\overline u}$ is morally null, hence $\text{Adj}(C)$ is morally null. Moreover, $\s_C\setminus\s_u=\mscr{H}(u|C)=\mscr{H}(u|\overline u)\sqcup\mscr{H}(\overline u|C)$, where $\mscr{H}(\overline u|C)$ is morally null; as a consequence, $\s_C$ is morally measurable. The same holds for $\mscr{H}(C)=\mscr{H}\setminus\left(\s_C\sqcup\s_C^*\right)$.
\end{proof}

We will denote by $\mscr{H}^0$ the collection of nowhere-dense halfspaces and set $\mscr{H}^{\x}:=\mscr{H}^0\cup\left(\mscr{H}^0\right)^*$. If $\mf{h}\in\mscr{H}\setminus\mscr{H}^{\x}$ we say that $\mf{h}$ is \emph{thick}.

\begin{cor}\label{nowhere-dense halfspaces}
If $X$ is locally convex, the set $\mscr{H}^{\x}$ is morally null.
\end{cor}
\begin{proof}
This follows from Lemma~\ref{adjacency} and the observation that, for every $x,y\in X$, we have $\mscr{H}^0\cap\mscr{H}(x|y)\cu\text{Adj}_y^*$. \end{proof}

\begin{prop}\label{separable}
Let $X$ be locally convex. If $X$ is separable, the measure space $\left(\mscr{H},\wh{\nu}\right)$ is $\s$-finite. The converse holds if $X$ has finite rank.
\end{prop}
\begin{proof}
If $\{x_n\}_{n\geq 0}$ is a countable dense subset of $X$, all halfspace-intervals $\mscr{H}(x_n|x_m)$ have finite measure and their union contains $\mscr{H}\setminus\mscr{H}^{\x}$; thus $\left(\mscr{H},\wh{\nu}\right)$ is $\s$-finite by Corollary~\ref{nowhere-dense halfspaces}.

Conversely, if $\left(\mscr{H},\wh{\nu}\right)$ is $\s$-finite, part~1 of Lemma~\ref{properties of nu hat} implies that there exists $\{x_n\}_{n\geq 0}\cu X$ such that the sets $\mscr{H}(x_n|x_m)$ cover $\mscr{H}$ up to a morally null set. By Lemma~6.4 in \cite{Bow1} and Corollary~\ref{intervals are compact new} above, hulls of separable subsets of $X$ are separable; thus, the hull $C$ of $\{x_n\}_{n\geq 0}$ is separable. If there existed a point $z\not\in\overline C$, the gate $\overline z$ for $(z,\overline C)$ would produce a positive-measure set $\mscr{H}(z|\overline z)$ disjoint from the union of the $\mscr{H}(x_n|x_m)$, contradiction. Thus $C$ is dense in $X$ and $X$ is separable.
\end{proof}

Recall that a subset $\mc{C}\cu\mscr{H}$ is said to be inseparable if it contains every halfspace $\mf{j}\in\mscr{H}$ such that there exist $\mf{h},\mf{k}\in\mc{C}$ with $\mf{h}\cu\mf{j}\cu\mf{k}$.

\begin{lem}\label{inseparable sets are morally measurable}
If $X$ has finite rank, any inseparable subset $\mc{C}\cu\mscr{H}$ is morally measurable. In particular, every filter $\s\cu\mscr{H}$ is morally measurable.
\end{lem}
\begin{proof}
It suffices to prove the lemma under the additional assumption that $\mc{C}\cu\mscr{H}(x|y)$, for points $x,y\in X$. By Lemma~\ref{Dilworth for differences} and Corollary~\ref{countable cofinal subsets}, $\mc{C}$ is a countable union of subsets of the form $\mscr{H}(\mf{k}^*|\mf{h})$, with $\mf{h},\mf{k}\in\mc{C}$. Each of these is morally measurable by Lemma~\ref{adjacency}.
\end{proof}

We can extend the notion of admissibility in Section~\ref{SMW} and \cite{CDH} as follows. We say that a partial filter $\s\cu\mscr{H}$ is \emph{tangible} if it is morally measurable and $\wh{\nu}(\s\setminus\s_x)<+\infty$ for some (equivalently, all) $x\in X$. For a morally measurable ultrafilter $\s$, tangibility is equivalent to having $\wh{\nu}(\s\triangle\s_x)<+\infty$, since $(\s_x\setminus\s)=(\s\setminus\s_x)^*$. For instance, all admissible ultrafilters $\s$ on the PMP $(\mscr{H},\mscr{B},\nu,\s_x)$ are tangible. Indeed, they are morally measurable since $\s=\s_x\sqcup\left[(\s\triangle\s_x)\setminus\s_x\right]$,
where $\s_x$ is morally measurable and $\s\triangle\s_x\in\mscr{B}$. 

We denote by $\mscr{M}(X)$ the set of tangible ultrafilters, identifying ultrafilters with $\wh{\nu}$-null symmetric difference. The analogy with the notation $\mc{M}(X)$ of Section~\ref{SMW} is justified by Corollary~\ref{tangible} below; for now, we simply observe that there are isometric embeddings $X\hookrightarrow\mc{M}(X)\hookrightarrow\mscr{M}(X)$ as a consequence of parts~2 and~3 of Lemma~\ref{properties of nu hat}. The following is a key result.

\begin{lem}\label{tangible lemma}
Let $X$ be locally convex. For every tangible filter $\s\cu\mscr{H}$, there exists $x\in X$ such that $\wh{\nu}(\s\setminus\s_x)=0$. 
\end{lem}
\begin{proof}
If there exists $x_0\in X$ such that $\wh{\nu}(\s\cap\mscr{H}(x_0|y))=0$ for all $y\in X$, then $\wh{\nu}(\s\setminus\s_{x_0})=0$. Indeed, given any $u,v\in X$, we can set $m:=m(x_0,u,v)$ and $(\s\setminus\s_{x_0})\cap\mscr{H}(u|v)=\s\cap\mscr{H}(m|v)\cu\s\cap\mscr{H}(x_0|v)$, the latter being morally null.

If instead $\wh{\nu}(\s\cap\mscr{H}(x_0|y))>0$ for some $x_0,y\in X$, then there exists $z\in I(x_0,y)$ such that $z\neq x_0$ and $\mscr{H}(x_0|z)\cu\s$. Indeed, by Lemma~\ref{adjacency} there exists $\mf{h}\in\s\cap\mscr{H}(x_0|y)$ such that $d(x_0,\mf{h})>0$. Letting $z$ be the gate for $(x_0,\overline{\mf{h}})$, if $\mf{k}\in\mscr{H}(x_0|z)$ we have $\mf{h}\cu\overline{\mf{h}}\cu\mf{k}$; hence, the fact that $\s$ is a filter and $\mf{h}\in\s$ implies that $\mf{k}\in\s$. 

Now, we construct a countable ordinal $\eta$ and an injective net $(x_{\alpha})_{\alpha\leq\eta}$ such that all the following are satisfied:
\begin{enumerate}
\item $(\s\setminus\s_{x_{\alpha+1}})\sqcup(\s_{x_{\alpha+1}}\setminus\s_{x_{\alpha}})=(\s\setminus\s_{x_{\alpha}})$ for each $\alpha\leq\eta$;
\item if $\alpha$ is a limit ordinal, $\bigcap_{\beta<\alpha}\s\setminus\s_{x_{\beta}}=\s\setminus\s_{x_{\alpha}}$ up to a morally null set;
\item $\wh{\nu}\left(\s\setminus\s_{x_{\eta}} \right)=0$.
\end{enumerate}
The net is constructed by transfinite induction, starting with an arbitrary choice of $x_0$. Suppose that $x_{\beta}$ has been defined for all $\beta<\alpha$. If $\alpha$ is a limit ordinal, the inductive hypothesis implies that the disjoint union of the sets $\s_{x_{\beta+1}}\setminus\s_{x_{\beta}}$, for $\beta<\alpha$, is contained in $\s\setminus\s_{x_0}$, up to a morally null set. Hence $\sum_{\beta<\alpha}d(x_{\beta},x_{\beta+1})\leq\wh{\nu}(\s\setminus\s_{x_0})<+\infty$ and $(x_{\beta})_{\beta<\alpha}$ is a Cauchy net. We define $x_{\alpha}$ to be its limit; we only need to check condition~2 and it follows from Lemma~\ref{adjacency}.

If $\alpha=\beta+1$, we look at $\s\setminus\s_{x_{\beta}}$; if this is morally null, we stop and set $\eta=\beta$. Otherwise, we can find $y\in X$ so that $\wh{\nu}\left(\s\cap\mscr{H}(x_{\beta}|y)\right)>0$ and we have already shown that there exists $z\in X$ with $\mscr{H}(x_{\beta}|z)\cu\s$; we set $x_{\alpha}:=z$. By construction $\s_{x_{\alpha}}\setminus\s_{x_{\beta}}\cu\s\setminus\s_{x_{\beta}}$ and, since $\s$ is a filter, we have $\s\cap(\s_{x_{\beta}}\setminus\s_{x_{\alpha}})=\emptyset$; in particular, $\s\setminus\s_{x_{\alpha}}\cu\s\setminus\s_{x_{\beta}}$. From this we immediately get condition~1, i.e.\ $\s\setminus\s_{x_{\beta}}=(\s\setminus\s_{x_{\alpha}})\sqcup(\s_{x_{\alpha}}\setminus\s_{x_{\beta}})$.

We conclude by remarking that, since $d\left(x_{\alpha+1},x_{\alpha}\right)>0$ for all ordinals $\alpha$ and $\sum d(x_{\beta},x_{\beta+1})\leq\wh{\nu}(\s\setminus\s_{x_0})<+\infty$, the process must terminate for some \emph{countable} ordinal $\eta$.
\end{proof}

\begin{cor}\label{tangible}
\begin{enumerate}
\item If $X$ is locally convex and $\s\cu\mscr{H}$ is a tangible ultrafilter, there exists $x\in X$ such that $\wh{\nu}(\s\triangle\s_x)=0$. 
\item If $X$ has finite rank and $\s\cu\mscr{H}$ is a tangible filter, there exists a convex subset $C\cu X$ such that $\wh{\nu}(\s\triangle\s_C)=0$.
\end{enumerate}
\end{cor}
\begin{proof}
Let $\s$ be a tangible filter and let $x\in X$ be such that $\wh{\nu}(\s\setminus\s_x)=0$, as provided by Lemma~\ref{tangible lemma}. If $\s$ is an ultrafilter, we immediately obtain $\wh{\nu}(\s\triangle\s_x)=0$. Otherwise, suppose that $X$ has finite rank and let $C$ be the intersection of all $\mf{h}\in\s\cap\s_x$; since $x\in C$, this is a nonempty convex subset and the filter $\s_C$ contains $\s\cap\s_x$. In particular, we have $\wh{\nu}(\s\setminus\s_C)=0$. 

Suppose $\mf{k}\in\s_C\setminus\s$ and consider the gate-projection $\pi\colon X\ra\overline{\mf{k}^*}$. Observe that no $\mf{h}\in\s\cap\s_x$ can be contained in $\mf{k}$ or we would have $\mf{k}\in\s$; hence every $\mf{h}\in\s\cap\s_x$ intersects $\mf{k}^*$. Part~1 of Lemma~\ref{gates vs inclusions} then implies that, given any $y\in C$, the projection $\pi(y)$ still lies in $C$; in particular, $C\cap\overline{\mf{k}^*}\neq\emptyset$. If $\mf{k}$ were open, we would have $\overline{\mf{k}^*}=\mf{k}^*$ and the previous statement would contradict the fact that $C\cu\mf{k}$; thus, $\mf{k}$ is closed by Corollary~\ref{open/closed}.

If $u\in X$ and $\overline u$ is its gate-projection to $\overline C$, we have $\overline u\in\overline{\mf{k}^*}$ for every $\mf{k}\in\s_C\setminus(\s\cup\s_u)$, since $C\cap\overline{\mf{k}^*}\neq\emptyset$. Since every such $\mf{k}$ is closed, we have $\overline u\in\overline{\mf{k}^*}\cap\mf{k}$ and $\s_C\setminus(\s\cup\s_u)\cu\text{Adj}_{\overline u}^*$. The arbitrariness of $u$ and Lemma~\ref{adjacency} imply that $\wh{\nu}(\s_C\setminus\s)=0$. 
\end{proof}

An immediate consequence of part~1 of Corollary~\ref{tangible} is:

\begin{cor}\label{M(X)=X}
For every complete, locally convex median space $X$, the isometric embedding $X\hookrightarrow\mscr{M}(X)$ is surjective. In particular, the spaces $X$, $\mc{M}(X)$ and $\mscr{M}(X)$ are isometric.
\end{cor}

An interesting consequence of Corollary~\ref{M(X)=X} is the following.

\begin{cor}
If $X$ is locally convex with corresponding SMH-structure $(X,\mscr{H},\mscr{B},\nu)$. Then $\text{Aut}(X,\mscr{H},\mscr{B},\nu)=Isom(X)$.
\end{cor}

An analogous result holds for the SMW-structure.

\section{Compactifying median spaces and algebras.}\label{the roller cptfication}

\subsection{The zero-completion of a median algebra.}\label{the zero-completion}

Let $M$ be a median algebra. We denote by $\mscr{J}(M)$ (or simply by $\mscr{J}$) the poset of all intervals $I(x,y)$ with $x,y\in M$, ordered by inclusion; singletons are allowed. Note that the poset $(\mscr{J},\cu)$ is not a directed set. Still, whenever $I,I',I''\in\mscr{J}$ and $I\cu I'\cu I''$, we have $\pi_I|_{I''}=\pi_I|_{I'}\o\pi_{I'}|_{I''}$ and it makes perfect sense to consider the inverse limit 
\[\varprojlim_{I\in\mscr{J}}I:=\mleft\{(x_I)_I\in\prod_{I\in\mscr{J}}I ~\middle | ~ \pi_{I\cap J}(x_I)=\pi_{I\cap J}(x_J), ~ \forall I,J\in\mscr{J},~ I\cap J\neq\emptyset\mright\} .\]
Note that $I\cap J\in\mscr{J}$ whenever $I,J\in\mscr{J}$; indeed, we showed in Section~\ref{median algebras} that every gate-convex subset of $I$ is itself an interval.

The product of all $I\in\mscr{J}$ has a structure of median algebra given by considering median maps component by component; the inverse limit $\varprojlim I$ also inherits a structure of median algebra and we have a monomorphism 
\begin{align*}
\iota\colon M & \hookrightarrow \varprojlim I \\
x & \mapsto (\pi_I(x))_I.
\end{align*}
\begin{defn}
We denote the median algebra $\varprojlim I$ simply by $\overline M$ and refer to it as the \emph{zero-completion} of $M$.
\end{defn}

This terminology comes from \cite{Bandelt-Meletiou} where the same notion is defined from a different perspective; more on this in Remark~\ref{original zero-completion}. 

We also have a monomorphism $i\colon M\hookrightarrow 2^{\mscr{H}}$ given by mapping $x\mapsto\s_x$; the space $2^{\mscr{H}}$ can be endowed with the product median algebra structure and the product topology. Here the set $2=\{0,1\}$ is equipped with the discrete topology and its unique structure of median algebra.

\begin{defn}\label{double dual}
The closure of $i(M)$ in $2^{\mscr{H}}$ with the induced median-algebra structure will be denoted by $M^{\o\o}$ and we will refer to it as the \emph{double dual} of $M$ (compare \cite{Roller}).
\end{defn}

\begin{lem}\label{double dual and zero-completion}
\begin{enumerate}
\item The median algebra $M^{\o\o}$ coincides with the subset of $2^{\mscr{H}}$ consisting of ultrafilters on $\mscr{H}$.
\item There is a monomorphism $j\colon\overline M\hookrightarrow M^{\o\o}$ such that $i=j\o\iota$.
\end{enumerate}
\end{lem}
\begin{proof}
Ultrafilters on $\mscr{H}$ form a closed subset of $2^{\mscr{H}}$ with the product topology; thus every element of $M^{\o\o}$ is an ultrafilter. Every neighbourhood of an ultrafilter $\s\cu\mscr{H}$ is of the form $\left\{A\cu\mscr{H}\mid \mf{h}_1,...,\mf{h}_k\in A,~ \mf{h}_{k+1}^*,...,\mf{h}^*_n\not\in A\right\}$, where $\mf{h}_1,...,\mf{h}_n$ lie in $\s$ and hence intersect pairwise. Any such neighbourhood intersects $i(M)$ as it contains $i(x)$ for any $x\in\mf{h}_1\cap ... \cap\mf{h}_n$; the latter is nonempty by Helly's Theorem. Hence, every ultrafilter on $\mscr{H}$ lies in $M^{\o\o}$.

Regarding part~2, given $(x_I)_I\in\varprojlim I$ we will construct an ultrafilter $\s\cu\mscr{H}$ such that $\s\cap\mscr{H}(I)=\s_{x_I}\cap\mscr{H}(I)$ for all $I\in\mscr{J}$. This ultrafilter is unique and the corresponding map $j\colon\overline M\ra M^{\o\o}$ is easily seen to be a monomorphism. It suffices to show that the sets
\begin{align*}
\Om_1&:=\{\mf{h}\in\mscr{H}\mid \exists I\text{ s.t. } x_I\in\mf{h}\in\mscr{H}(I)\}, \\
\Om_2&:=\{\mf{h}\in\mscr{H}\mid x_I\in\mf{h}, \forall I \text{ s.t. } \mf{h}\in\mscr{H}(I)\},
\end{align*}
coincide and set $\s:=\Om_1=\Om_2$. This is indeed an ultrafilter: $\Om_1$ contains at least one side of every wall of $M$, while any two halfspaces in $\Om_2$ intersect.

The inclusion $\Om_2\cu\Om_1$ is immediate; proving the other amounts to showing that $x_J\in\mf{h}\in\mscr{H}(J)$ and $x_I\not\in\mf{h}\in\mscr{H}(I)$ cannot happen at the same time. We argue by contradiction. 

Observe that $m_I:=\pi_I(x_J)\in\mf{h}$ and $m_J:=\pi_J(x_I)\in\mf{h}^*$ by part~1 of Lemma~\ref{gates vs inclusions}. Let $I':=I(x_I,m_I)$ and $J':=I(x_J,m_J)$; since $x_I\in I'\cu I$, we have $x_{I'}=\pi_{I'}(x_I)=x_I$ and similarly $x_{J'}=x_J$. Set $K:=I(x_I,x_J)$. Since $I'\cu K$ and $J'\cu K$, we have $\pi_{I'}(x_K)=x_I$ and $\pi_{J'}(x_K)=x_J$. However, $\{\mf{h},\mf{h}^*\}\in\mscr{W}(I')\cap\mscr{W}(J')$, so the previous equalities imply that $x_K\in\mf{h}^*$ and $x_K\in\mf{h}$, respectively, a contradiction.
\end{proof}

\begin{cor}\label{recognising gate-projections}
\begin{enumerate}
\item The embedding $\iota\colon M\ra\overline M$ has convex image. Thus, given $x,y\in M$, the notion of $I(x,y)$ is the same in $M$ and $\overline M$.
\item For every $J\in\mscr{J}$, the projection $p_J\colon\overline M\ra J$ that $\overline M$ inherits from $\prod I$ is precisely the gate-projection $\pi_J\colon\overline M\ra J$.
\end{enumerate}
\end{cor}
\begin{proof}
Consider points $x,y\in M$ and set $J:=I(x,y)\in\mscr{J}$. Given $\underline z\in\overline M$, we write $\underline z_I$ instead of $p_I(\underline z)$. Assuming $m(\iota(x),\iota(y),\underline z)=\underline z$, we will show that $\underline z=\iota(\underline z_J)$. This yields $\underline z\in\iota(M)$ and proves part~1.

If we had $\underline z\neq\iota(\underline z_J)$, there would exist $I\in\mscr{J}$ such that $\underline z_I\neq p_I\iota(\underline z_J)$; note that $p_I\iota(\underline z_J)=\pi_I(\underline z_J)$, where $\pi_I$ denotes the gate-projection $M\ra I$. Let $\mf{h}\in\mscr{H}(M)$ be a halfspace lying in $\mscr{H}(\pi_I(\underline z_J)|\underline z_I)$; by Proposition~\ref{walls in convex}, we have $\mf{h}\in\mscr{H}(\underline z_J|\underline z_I)$. The equality $\Om_1=\Om_2$ in the proof of Lemma~\ref{double dual and zero-completion} now shows that $\mf{h}\not\in\mscr{H}(J)$. Since $\underline z_J\in\mf{h}^*$, we have $J\cu\mf{h}^*$; in particular, $\pi_I(x)$ and $\pi_I(y)$ lie in $\mf{h}^*$. However, since $m(\iota(x),\iota(y),\underline z)=\underline z$, we have $m(\pi_I(x),\pi_I(y), \underline z_I)=p_Im(\iota(x),\iota(y),\underline z)=\underline z_I\in\mf{h}$, a contradiction.

For part~2, note that $\pi_J(\underline z)=m(x,y,\underline z)$ lies in $\iota(M)$ by part~1. Thus, 
\[m(x,y,\underline z)=p_Jm(x,y,\underline z)=m(p_J(x),p_J(y),p_J(\underline z))=m(x,y,\underline z_J)=\underline z_J.\]
\end{proof}

The median algebras $\overline M$ and $M^{\o\o}$ can coincide; for instance, this is the case for $0$-skeleta of a ${\rm CAT}(0)$ cube complexes. However, the following example shows that $\overline M$ and $M^{\o\o}$ differ in general.

\begin{ex}\label{double dual of N}
Consider the median algebras $N=\N$ and $M=\N\cup\{+\infty\}$; in both cases $m(x,y,z)=y$ if $x\leq y\leq z$. For every $k\in\N$, both $M$ and $N$ have a wall $\mf{w}_k$ separating $k$ and $k+1$; in $M$, there is an additional wall $\mf{w}_{\infty}$ separating $\N$ and $+\infty$. Observe that $M=\overline N=N^{\o\o}$ and, since $M=I(0,+\infty)$, we also have $M=\overline M$. 

However, $M^{\o\o}=M\cup\{\infty_-\}$, where $\infty_-$ is represented by the ultrafilter that picks the side containing $+\infty$ for every wall $\mf{w}_k$ and the side containing $\N$ for the wall $\mf{w}_{\infty}$. The point $\infty_-$ is ``bigger than any natural number'', but still ``smaller than $+\infty$''; also compare Remark~\ref{bad ultrafilters on R}. 
\end{ex}

Given $a\in M$, we say that a convex subset $C\cu M$ is \emph{$a$-directed} if $a\in C$ and, for every $x,y\in C$, there exists $z\in C$ such that $x,y\in I(a,z)$.

\begin{lem}\label{recognising the zero-completion}
Fix $a\in M$. There is a one-to-one correspondence between points of $\overline M$ and gate-convex, $a$-directed subsets $C\cu M$. Points $b\in M\cu\overline M$ correspond to intervals $I(a,b)$.
\end{lem}
\begin{proof}
If $C\cu M$ is gate-convex and $a$-directed, the projection $C_I:=\pi_I(C)$ is a gate-convex, $\pi_I(a)$-directed subset of $I$ by part~2 of Lemma~\ref{gates vs inclusions}. It follows that there exists a (unique) point $x_I\in C_I$ such that $C_I=I(\pi_I(a),x_I)$. By Proposition~\ref{retractions new}, we obtain a point $\xi_C:=(x_I)_I\in\varprojlim I$.

Conversely, given $\xi\in\overline M$, we can consider the interval $I(a,\xi)\cu\overline M$ and set $C_{\xi}:=I(a,\xi)\cap M$. Since $M$ is convex in $\overline M$, the map $u\mapsto m(a,\xi,u)$ takes $M$ into itself and it is a gate-projection $M\ra C_{\xi}$ by Proposition~\ref{retractions new}. If $x,y\in C_{\xi}$, we have $x,y\in I(a,z)$ with $z:=m(x,y,\xi)\in C_{\xi}$. Thus $C_{\xi}$ is gate-convex and $a$-directed. 

Observe that, setting $z:=m(a,\xi,\pi_I(\xi))$, we have $\pi_I(z)=\pi_I(\xi)$ and $z\in I(a,\xi)\cap M$; hence, $\pi_I(I(a,\xi)\cap M)=I(\pi_I(a),\pi_I(\xi))$ for all $I\in\mscr{J}$, i.e.\ $\xi=\xi_{C_{\xi}}$ for every $\xi\in\overline M$. We conclude by observing that $C_1\neq C_2$ implies $\xi_{C_1}\neq\xi_{C_2}$; indeed, if $x\in C_1\setminus C_2$ and $J:=I(a,x)$, we have $\pi_J(\xi_{C_1})=x$ and $\pi_J(\xi_{C_1})\neq x$, since $\pi_J(C_1)=J$ and $\pi_J(C_2)=C_2\cap J\neq J$.
\end{proof}

\begin{rmk}\label{original zero-completion}
When zero-completions were originally defined in \cite{Bandelt-Meletiou}, points of $\overline M$ were compatible meet-semilattice operations, see Theorem~1 in \emph{op.\ cit.}. The terminology ``zero-completion'' is due to the fact that all semilattice operations have a zero in $\overline M$, while they need not have one in $M$. By Lemma~\ref{recognising the zero-completion} above and Theorem~5.5 in \cite{Bandelt-Hedlikova}, our definition of $\overline M$ yields the same object. Zero-completions do not seem to have been studied outside of \cite{Bandelt-Meletiou}; we will develop their theory further in the next sections, especially in the case when $M$ arises from a median space. 
\end{rmk}

\begin{lem}\label{C bar in M bar}
If $C\cu M$ is gate-convex, the zero-completion $\overline C$ canonically embeds into $\overline M$ as a gate-convex subset with $\overline C\cap M=C$.
\end{lem}
\begin{proof}
If $\xi=(x_I)_I\in\varprojlim I$, we set $\pi(\xi):=(\pi_I\pi_C(x_I))_I\in\prod I$. This is an element of $\varprojlim I$ as, if $J\cu K$ are intervals of $M$, we have
\[\pi_J(\pi_K\pi_C(x_K))=\pi_J\pi_C\pi_K(\xi)=\pi_J\pi_C\pi_J\pi_K(\xi)=\pi_J\pi_C\pi_J(\xi)=\pi_J\pi_C(x_J),\]
by Corollary~\ref{recognising gate-projections} and part~4 of Lemma~\ref{gates vs inclusions}. We obtain a map $\pi\colon\overline M\ra\overline M$, which is a gate-projection onto its image, by Proposition~\ref{retractions new}. The image of $\pi$ is the set of those $(x_I)_I$ with $\pi_I\pi_C(x_I)=x_I$, i.e.\ $x_I\in\pi_I(C)$, for all $I\in\mscr{J}$. Restricting the index set to $\mscr{J}(C)$, we obtain a morphism $f\colon\text{im}~\pi\ra\overline C$. Embedding $\overline M\hookrightarrow M^{\o\o}$ as in Lemma~\ref{double dual and zero-completion}, all points of $\text{im}~\pi$ are represented by ultrafilters containing $\s_C$. In terms of ultrafilters, $f$ is the restriction of the map that takes each ultrafilter on $\mscr{H}(M)$ to its intersection with the subset $\mscr{H}(C)\cu\mscr{H}(M)$; in particular, $f$ is injective.

We now show that $f$ is surjective; given $\xi\in\overline C$, let $\s\cu\mscr{H}(C)$ be the ultrafilter representing $\xi$, as provided by Lemma~\ref{double dual and zero-completion}. If $I\in\mscr{J}(M)$, there exist points $u,v\in C$ such that $\pi_I(C)=I(\pi_I(u),\pi_I(v))$; we set $J:=I(u,v)$. If $\xi_J$ is the coordinate of $\xi$ corresponding to $J\in\mscr{J}(C)$, we set $x_I:=\pi_I(\xi_J)$. Note that $x_I$ is represented by the ultrafilter $(\s_C\cap\mscr{H}(I))\sqcup (\s\cap\mscr{H}(I))$ on $\mscr{H}(I)$. In particular, our definition of $x_I$ does not depend on the choice of the points $u,v$ and $(x_I)_I$ satisfies the compatibility condition necessary to define a point $\eta\in\overline M$. It is clear that $\eta\in\text{im}~\pi$ and $f(\eta)=\xi$. Thus, we have identified $\overline C\simeq\text{im}~\pi$; the fact that $\overline C\cap M=C$ is a trivial observation.
\end{proof}

Given $\mf{h}\in\mscr{H}(M)$, there exists a unique halfspace $\wt{\mf{h}}\in\mscr{H}(\overline M)$ satisfying $\wt{\mf{h}}\cap M=\mf{h}$. This can be observed by applying Proposition~\ref{walls in convex} to an interval $I(x,y)$ with $x\in\mf{h}$ and $y\in\mf{h}^*$. 

The halfspace $\wt{\mf{h}}$ can be further characterised by ${\xi\in\wt{\mf{h}}\Leftrightarrow\mf{h}\in j(\xi)}$, where $j$ is as in Lemma~\ref{double dual and zero-completion}. Indeed, Corollary~\ref{recognising gate-projections} implies that, for every interval $I\cu M$, we have $j(\xi)\cap\mscr{H}(I)=\s_{\pi_I(\xi)}\cap\mscr{H}(I)$. If $\mf{h}\in\mscr{H}(I)$, Proposition~\ref{walls in convex} then yields
\[\xi\in\wt{\mf{h}}\Leftrightarrow\pi_I(\xi)\in\mf{h}\Leftrightarrow\mf{h}\in j(\xi).\]

Note that not all halfspaces of $\overline M$ are of the form $\wt{\mf{h}}$; indeed, by convexity of $M$, we have $\mscr{W}(M|\xi)\neq\emptyset$ for every $\xi\in\overline M\setminus M$. An explicit example of such a wall was given in Example~\ref{double dual of N} for the median algebra $N$ (in the notation of the example, we have $\mf{w}_{\infty}\in\mscr{W}(N|{+\infty})$ and $+\infty\in\partial N$).

Nevertheless, we can identify the set of halfspaces of $M$ with a subset $\mscr{H}(M)\cu\mscr{H}(\overline M)$ and any two points of $\overline M$ are separated by an element of $\mscr{H}(M)$. Indeed, if $\xi,\eta\in\overline M$ are distinct, then $\pi_I(\xi)\neq\pi_I(\eta)$ for some interval $I\cu M$ and, for any $\mf{h}\in\mscr{H}(\pi_I(\xi)|\pi_I(\eta))$, we have $\wt{\mf{h}}\in\mscr{H}(\xi|\eta)$.

\begin{lem}\label{rk M bar}
We have $\text{rank}(\overline M)=\text{rank}(M)$.
\end{lem}
\begin{proof}
Immediate from the discussion above and Lemma~\ref{rank with a subset}.
\end{proof}

\begin{lem}
If there exists a topology on $M$ for which it is a compact topological median algebra, then $M=\overline M$.
\end{lem}
\begin{proof}
Given $\xi=(x_I)_I\in\varprojlim I$, consider the projections $\pi_I\colon\overline M\ra I$ and set $C_I:=\pi_I^{-1}(x_I)\cu\overline M$. These are convex sets by Proposition~\ref{retractions new} and they pairwise intersect as they all contain $\xi$; moreover, they all intersect $M$, which is convex in $\overline M$. Helly's Theorem implies that $\{C_I\cap M\mid I\in\mscr{J}\}$ has the finite intersection property.

Now endow $M$ with its compact topology. Each $C_I\cap M=(\pi_I|_M)^{-1}(x_I)$ is closed in $M$, hence compact. Thus, the intersection of all the $C_I\cap M$ is nonempty; any $x$ in this intersection satisfies $\pi_I(x)=x_I$ for every interval $I$, i.e.\ $\xi=\iota(x)$.
\end{proof}

In the rest of the section, we suppose that $M$ is a topological median algebra. Endowing the product of all intervals with the product topology, the zero-completion $\overline M$ inherits a topology. The inclusion $\iota\colon M\hookrightarrow\overline M$ is always continuous, but in general not a topological embedding (cf.\ Proposition~\ref{boundary is closed if loc cpt}). Classical examples of this phenomenon are provided by locally infinite trees, since in that case $\overline M$ coincides with the usual Roller compactification.

For instance, if $T$ is a geodesically complete (real or simplicial) tree and $x\in T$ is a point of infinite degree, $\iota(x)$ does not lie in the interior of $\iota(T)$; in particular, $\iota(T)$ is not open in $\overline T$. Another interesting case is when $T$ is a bounded tree and $T\setminus\{x\}$ contains infinitely many connected components of diameter at least $1$; here the map $\iota\colon T\ra\overline T$ is a continuous bijection, but still not a homeomorphism.

The following is an easy observation.

\begin{lem}\label{zero-completion locally convex}
If $M$ is locally convex, $\overline M$ is as well.
\end{lem}

\begin{lem}\label{M dense}
If $M$ has compact intervals, $\overline M$ is compact and $M$ is dense in $\overline M$.
\end{lem}
\begin{proof}
Compactness is immediate from the observation that $\overline M$ is a closed subset of $\prod I$. We prove density by showing that for every $\xi\in\overline M$ and for every finite collection of intervals $I_1,...,I_k\in\mscr{J}(M)$, there exists $x\in M$ such that $\pi_{I_i}(x)=\pi_{I_i}(\xi)$ for all $i$. 

Let $C$ be the convex hull of $I_1\cup ...\cup I_k$ in $\overline M$; it is compact by Lemma~\ref{compact median algebras} and $C\cu M$ since $M$ is convex in $\overline M$. Note that $C$ is gate-convex in $\overline M$ by Lemma~\ref{gate-convexity vs compact intervals}; let $\pi\colon\overline M\ra C$ be the corresponding gate-projection. By part~3 of Lemma~\ref{gates vs inclusions}, we can set $x:=\pi(\xi)\in C\cu M$.
\end{proof}

If $M$ has compact intervals and $C\cu M$ is gate-convex, Lemma~\ref{M dense} implies that the subset of $\overline M$ that we identified with the zero-completion $\overline C$ in Lemma~\ref{C bar in M bar} coincides with the closure of $C$ in the topology of $\overline M$.

\subsection{The Roller compactification of a median space.}\label{definition and first properties}

Let $X$ be a complete, locally convex median space with compact intervals throughout this section. This encompasses all finite rank median spaces (see Corollary~\ref{intervals are compact new}), all (possibly infinite dimensional) ${\rm CAT}(0)$ cube complexes and all (possibly infinite rank) complete, connected, locally compact, locally convex median spaces (see \cite{Chatterji-Drutu} for examples).

\begin{defn}
In this context, we will refer to the zero-completion $\overline X$ as \emph{Roller compactification} and to $\partial X:=\overline X\setminus X$ as \emph{Roller boundary}.
\end{defn}

The renaming is justified by the fact that the median metric on $X$ induces an additional structure on $\overline X$, which we shall study in this and the following section. There is a strong analogy with Roller boundaries of ${\rm CAT}(0)$ cube complexes and, indeed, if $X$ is the $0$-skeleton of a ${\rm CAT}(0)$ cube complex, our notion of Roller boundary coincides with the usual one. The following proposition sums up what we already know about $\overline X$; it roughly corresponds to the first and third definitions of $\overline X$ that we gave in the introduction.

\begin{thm}\label{properties of X bar}
\begin{enumerate}
\item The Roller compactification $\overline X$ is a locally convex, compact, topological median algebra. 
\item The inclusion $\iota\colon X\hookrightarrow\overline X$ is a continuous morphism with convex, dense image. 
\item For every closed convex subset $C\cu X$, the closure of $C$ in $\overline X$ is gate-convex and naturally identified with the Roller compactification of $C$.
\item If $X$ is separable, the topology of $\overline X$ is separable and metrisable.
\end{enumerate}
\end{thm}

We remark that $\partial X$ need not be closed, compare Proposition~\ref{boundary is closed if loc cpt} below. In analogy with the space $\mscr{M}(X)$ and Definition~\ref{double dual}, we introduce
\[\overline{\mscr{M}}(X):=\left\{\s\cu\mscr{H}\mid \s\in\wh{\mscr{B}},~\s\text{ ultrafilter}\right\}\Big/{\sim}~,\]
where $\s_1\sim\s_2$ if $\wh{\nu}(\s_1\triangle\s_2)=0$. We can give $\overline{\mscr{M}}(X)$ a median-algebra structure by defining the median map as in part~2 of Example~\ref{L1 are median}. If $X$ has finite rank, all ultrafilters are morally measurable by Lemma~\ref{inseparable sets are morally measurable} and $\overline{\mscr{M}}(X)$ is simply a quotient of $X^{\o\o}$. By Corollary~\ref{M(X)=X}, we have a monomorphism $X\simeq\mscr{M}(X)\hookrightarrow\overline{\mscr{M}}(X)$. The following result corresponds to the second definition of the Roller compactification that we gave in the introduction.

\begin{thm}\label{to highlight}
The map $j\colon\overline{X}\hookrightarrow X^{\o\o}$ introduced in Lemma~\ref{double dual and zero-completion} takes values in the set of morally measurable ultrafilters and it descends to an isomorphism $\overline j\colon\overline X\xrightarrow{\simeq}\overline{\mscr{M}}(X)$ extending $X\hookrightarrow\overline{\mscr{M}}(X)$.
\end{thm}
\begin{proof}
If $\xi=(x_I)_I\in\varprojlim I$, we have $j(\xi)\cap\mscr{H}(I)=\s_{x_I}\cap\mscr{H}(I)$ for every interval $I$, hence $j(\xi)$ is morally measurable. We get a morphism ${\overline j\colon\overline X\ra\overline{\mscr{M}}(X)}$ extending $X\hookrightarrow\overline{\mscr{M}}(X)$. If $\eta=(y_I)_I$ satisfies $\wh{\nu}(j(\xi)\triangle j(\eta))=0$, we have $\wh{\nu}((\s_{x_I}\triangle\s_{y_I})\cap\mscr{H}(I))=0$, i.e.\ $x_I=y_I$, for all $I\in\mscr{J}$. Thus, $\overline j$ is injective. 

If $\s\cu\mscr{H}$ is a morally measurable ultrafilter, each $\s\cap\mscr{H}(I)$ is a morally measurable ultrafilter on $\mscr{H}(I)$ and it is tangible since $\wh{\nu}(\mscr{H}(I))<+\infty$. Corollary~\ref{M(X)=X} provides $z_I\in I$ with $\wh{\nu}((\s_{z_I}\triangle\s)\cap\mscr{H}(I))=0$. We obtain $\zeta:=(z_I)_I\in\varprojlim I$ with $\wh{\nu}(j(\zeta)\triangle\s)=0$; hence, $\overline j$ is also surjective. 
\end{proof}

In general, given $\xi\in\overline X$ and a morally measurable ultrafilter $\s\cu\mscr{H}$ representing $\xi$, we could have $\mf{h}\in\s$ even if $\xi\not\in\wt{\mf{h}}$; see e.g.\ Remark~\ref{bad ultrafilters on R}. However, we have already observed that this does not happen for ${\s=j(\xi)}$. Since $j(x)=\s_x$ for every $x\in X$, we will also denote $j(\xi)$ by $\s_{\xi}$ from now on. This should be viewed as a canonical choice of an ultrafilter representing $\xi$.

\begin{lem}\label{two convergences}
A sequence $(\xi_n)_{n\geq 0}$ in $\overline X$ converges to $\xi\in\overline X$ if and only if $\wh{\nu}\left(\limsup \s_{\xi}\triangle\s_{\xi_n}\right)=0$. 
\end{lem}
\begin{proof}
Given $I\in\mscr{J}$, Lemma~\ref{adjacency} implies that $\pi_I(\xi_n)\ra\pi_I(\xi)$ if and only if 
\[0=\wh{\nu}\left(\limsup_{n\ra+\infty}\s_{\pi_I(\xi)}\triangle\s_{\pi_I(\xi_n)}\right)=\wh{\nu}\left(\limsup_{n\ra+\infty}(\s_{\xi}\triangle\s_{\xi_n})\cap\mscr{H}(I)\right).\]
Since $\xi_n\ra\xi$ if and only if $\pi_I(\xi_n)\ra\pi_I(\xi)$ for all $I\in\mscr{J}$, convergence corresponds precisely to $\limsup \s_{\xi}\triangle\s_{\xi_n}$ being morally null.
\end{proof}

We can endow $\overline X\simeq\overline{\mscr{M}}(X)$ with an \emph{extended metric}
\[d(\s_1,\s_2):=\frac{1}{2}\cdot\wh{\nu}(\s_1\triangle\s_2)\in [0,+\infty];\]
the restriction to $\mscr{M}(X)\simeq X$ coincides with the usual metric on $X$.

\begin{lem}\label{gate-projections are 1-Lip}
\begin{enumerate}
\item The median map $m\colon\overline X^3\ra\overline X$ is $1$-Lipschitz with respect to the extended metric $d$.
\item If $C\cu\overline X$ is closed and convex, the gate-projection $\pi\colon\overline X\ra C$ is $1$-Lipschitz with respect to the extended metric $d$.
\end{enumerate}
\end{lem}
\begin{proof}
Part~1 follows from from part~2 applied to intervals. Denote by $\s_C$ the set of $\mf{h}\in\mscr{H}$ such that $C\cu\wt{\mf{h}}$ and set $\mscr{H}(C):=\mscr{H}\setminus(\s_C\sqcup\s_C^*)$. By part~1 of Lemma~\ref{gates vs inclusions}, we have $\s_{\pi(\xi)}=(\s_{\xi}\cap\mscr{H}(C))\sqcup\s_C$, for all $\xi\in\overline X$. Thus, for all $\xi,\eta\in\overline X$,
\[2\cdot d(\pi(\xi),\pi(\eta))=\wh{\nu}\left((\s_{\xi}\triangle\s_{\eta})\cap\mscr{H}(C)\right)\leq\wh{\nu}\left(\s_{\xi}\triangle\s_{\eta}\right)=2\cdot d(\xi,\eta).\]
\end{proof}

If $\xi\in\overline X$ and $(\xi_n)_{n\geq 0}$ is a sequence in $\overline X$ with $d(\xi_n,\xi)\ra 0$, part~2 of Lemma~\ref{gate-projections are 1-Lip} shows that ${\pi_I(\xi_n)\ra\pi_I(\xi)}$ for every interval $I\cu X$; hence $\xi_n\ra\xi$ in $\overline X$. The converse does not hold: consider a sequence $(x_n)_{n\geq 0}$ in $X\cu\overline X$ that converges to a point $\xi\in\partial X$ (these exist by Lemma~\ref{M dense}). Corollary~\ref{M(X)=X} shows that $d(x_n,\xi)=+\infty$ for all $n\geq 0$, so $d(x_n,\xi)\not\ra 0$.

The following notion is due to \cite{Guralnik} for ${\rm CAT}(0)$ cube complexes.

\begin{defn}\label{connected components}
A \emph{component} of $\overline X$ is a $\approx$-equivalence class of morally measurable ultrafilters for the equivalence relation
\[\s_1\approx\s_2\xLeftrightarrow{\text{def}^{\underline{\text{n}}}} d(\s_1,\s_2)<+\infty.\]
\end{defn}

Note that the subset $X\cu\overline X$ always forms a single component. Part~2 of Example~\ref{L1 are median} implies the following.

\begin{prop}\label{cc's are median and convex}
The restriction of the metric $d$ to any component of $\overline X$ gives it a structure of median space. Each component is a convex in $\overline X$.
\end{prop}

The study of components of $\partial X$ will be the subject of Section~\ref{cc's}. 

\begin{prop}\label{boundary is closed if loc cpt}
If $X$ is connected and locally compact, the inclusion $\iota\colon X\ra\overline X$ is a topological embedding.
\end{prop}
\begin{proof}
Given $x_0\in X$, choose $0<\delta<+\infty$ and let $F$ be a finite, $\delta$-dense subset of $\overline{B}_{2\delta}(x_0)$; this exists since $X$ is proper by Proposition~3.7 in Chapter~I.3 of \cite{BH}. Let $\mc{U}$ be the set of points $\xi\in\overline X$ such that $d(\pi_I(\xi),x_0)<\delta$ for all intervals $I=I(x_0,y)$, with $y\in F$; this is a neighbourhood of $x_0$ in $\overline X$. If $\eta\in\overline X$ and $d(\eta,x_0)\geq 2\delta$, there exists $z\in I(x_0,\eta)$ with $d(x_0,z)=2\delta$, since $X$ is geodesic. Choosing $y\in F$ with $d(y,z)<\delta$ we have
\[d(m(\eta,x_0,y),x_0)>d(m(\eta,x_0,z),x_0)-\delta=d(z,x_0)-\delta=\delta.\]
Thus $\mc{U}\cu B_{2\delta}(x_0)\cu X$ and $x_0\in\mc{U}$. Since $x_0$ and $\delta$ were arbitrary, this shows that the map $\iota$ is open.
\end{proof}

Note that connectedness cannot be dropped from the statement of Proposition~\ref{boundary is closed if loc cpt}. For instance, let $T$ be the tree obtained by glueing countably many rays $\{r_n\mid n\in\N\}$ by their origins. Consider the closed subset $X\cu T$ obtained by removing the open interval $(0,n)\cu [0,+\infty)$ from the ray $r_n$, for each $n\geq 0$. The median space $X$ is proper, but the inclusion $\iota\colon X\ra\overline X$ does not have open image; indeed, in $\overline X$ the endpoints at infinity of the rays $r_n$ converge to their common origin.

We conclude this section by presenting one more characterisation of $\overline X$. Fixing $x_0\in X$, we denote by $\mc{C}_{Lip}(X)_{x_0}$ the set of $1$-Lipschitz functions $X\ra\R$ taking $x_0$ to $0$; we endow this space with the topology of pointwise convergence, which is compact
 and coincides with the topology of uniform convergence on compact subsets. The map 
\begin{align*}
B_{x_0}\colon X & \hookrightarrow\mc{C}_{Lip}(X)_{x_0} \\
x & \mapsto d(x,\cdot)-d(x,x_0);
\end{align*}
is continuous and it is customary to refer to $\overline{B_{x_0}(X)}$ as the \emph{horofunction compactification} (or \emph{Busemann compactification}) of $X$; indeed, it does not depend on the basepoint $x_0$. The following is an extension of an unpublished result of U.\ Bader and D.\ Guralnik in the case of ${\rm CAT}(0)$ cube complexes (see e.g.\ the appendix to \cite{Caprace-Lecureux}).

\begin{prop}\label{Busemann compactification}
The identity map of $X$ extends to a homeomorphism between its Roller and Busemann compactifications.
\end{prop}
\begin{proof}
Since $B_{x_0}(x)[z]=d(z,m(z,x,x_0))-d(x_0,m(z,x,x_0))$, we can construct an extension of $B_{x_0}$ taking values in the space ${\R}^X_{x_0}$ of functions $X\ra\R$ taking $x_0$ to $0$:
\begin{align*}
\wt{B}_{x_0}\colon\overline X & \longrightarrow {\R}^X_{x_0} \\
\xi & \longmapsto d(\cdot,m(\cdot,\xi,x_0))-d(x_0,m(\cdot,\xi,x_0)).
\end{align*}
This is well-defined due to the convexity of $X\cu\overline X$. If $\xi,\eta\in\overline X$ and ${\xi\neq\eta}$, there exists $\mf{h}\in\mscr{H}$ with $\wt{\mf{h}}\in\mscr{H}(\xi|\eta)$. Without loss of generality, we can assume that $x_0\in\mf{h}^*$. Pick a point $x\in\mf{h}$ and set ${u:=m(x_0,x,\eta)}$, ${v:=m(x_0,u,\xi)}$; since $\mf{h}\in\mscr{H}(v|u)$, we have $u\neq v$. In particular, 
\[\wt{B}_{x_0}(\xi)[u]=d(u,v)-d(x_0,v)>-d(x_0,v)>-d(x_0,u)=\wt{B}_{x_0}(\eta)[u].\]
This shows that $\wt{B}_{x_0}$ is injective; we now prove that it is continuous for the topology of pointwise convergence. Given $\xi\in\overline X$, $x\in X$ and $\eps>0$, there exists a neighbourhood $U$ of $\xi$ such that, for every $\eta\in U$, the projections of $\xi$ and $\eta$ to $I(x_0,x)$ are at distance smaller than $\eps/2$. In particular, for $\eta\in U$, the triangle inequality yields
\[\left|\wt{B}_{x_0}(\xi)[x]-\wt{B}_{x_0}(\eta)[x]\right|\leq 2\cdot d(m(x,\xi,x_0),m(x,\eta,x_0))<\eps.\]
Continuity of $\wt{B}_{x_0}$ and Lemma~\ref{M dense} imply that $\wt{B}_{x_0}(\overline X)$ coincides with the Busemann compactification. Finally, since $\overline X$ is compact, $\wt{B}_{x_0}$ is a closed map, hence a homeomorphism.
\end{proof}

As a consequence of the proof of Proposition~\ref{Busemann compactification}, we can define $1$-Lipschitz Busemann functions for points in the Roller boundary.

\begin{cor}
For every $\xi\in\overline X$ and $x_0\in X$, the function $X\ra\R$ defined by
\[z\longmapsto d(z,m(z,\xi,x_0))-d(x_0,m(z,\xi,x_0))\]
is $1$-Lipschitz.
\end{cor}

\subsection{Components of the Roller boundary.}\label{cc's}

Let $X$ be a complete, locally convex median space with compact intervals. In this section, we study the structure of the median spaces arising as components of $\partial X$. Our first goal is to obtain the following.

\begin{prop}\label{cc's complete}
Components of $\overline X$ are complete. 
\end{prop}

To do so, we need to relate the extended metric on $\overline X$ to its restriction to the intervals of $X$. 

\begin{prop}\label{metric via intervals}
For every $\xi,\eta\in\overline X$, we have
\[d(\xi,\eta)=\sup_{I\in\mscr{J}(X)} d\left(\pi_I(\xi),\pi_I(\eta)\right).\]
\end{prop}
\begin{proof}
The inequality $\geq$ follows from Lemma~\ref{gate-projections are 1-Lip}. Given $\eps>0$, we will produce an interval $I\cu X$ with $d\left(\pi_I(\xi),\pi_I(\eta)\right)\geq d(\xi,\eta)-\eps$. By the definition of $\wh{\nu}$, there exist points $x_1,...,x_n,y_1,...,y_n\cu X$ such that $\mscr{H}(x_i|y_i)$ are pairwise-disjoint and
\[\sum_{k=1}^n\wh{\nu}\left((\s_{\xi}\setminus\s_{\eta})\cap\mscr{H}(x_i|y_i)\right)\geq d(\xi,\eta)-\eps.\] 
Suppose that $n$ is minimal among the integers for which such an inequality holds; we will show that $n=1$, which will conclude the proof. Suppose for the sake of contradiction that $n\geq 2$ and set ${u:=m(\eta,x_1,x_2)}$, ${v:=m(\xi,y_1,y_2)}$. Observe that 
\[(\s_{\xi}\setminus\s_{\eta})\cap\left(\mscr{H}(x_1|y_1)\sqcup\mscr{H}(x_2|y_2)\right)\cu(\s_{\xi}\setminus\s_{\eta})\cap\mscr{H}(u|v),\]
which, applying $\wh{\nu}$, violates the minimality of $n$. 
\end{proof}

\begin{proof}[Proof of Proposition~\ref{cc's complete}]
Let $(\xi_n)_{n\geq 0}$ be a Cauchy sequence in a component of the Roller boundary. By Lemma~\ref{gate-projections are 1-Lip}, the sequence $\pi_I(\xi_n)$ is also Cauchy for every $I\in\mscr{J}$ and it has a limit $\xi_{I}\in I$. These points define a point ${\xi:=(\xi_I)_I\in\varprojlim I}$. 
By Proposition~\ref{metric via intervals},
\begin{align*}
d(\xi,\xi_n)=\sup_{I\in\mscr{J}} d\left(\xi_I,\pi_I(\xi_n)\right)&=\sup_{I\in\mscr{J}}\lim_{m\ra +\infty} d\left(\pi_I(\xi_m),\pi_I(\xi_n)\right) \\
&\leq\sup_{I\in\mscr{J}}\limsup_{m\ra +\infty} d\left(\xi_m,\xi_n\right)=\limsup_{m\ra +\infty} d\left(\xi_m,\xi_n\right)
\end{align*}
and the latter converges to zero as $n$ goes to infinity. 
\end{proof}

\begin{prop}\label{cc's compact intervals}
Components of $\overline X$ have compact intervals. 
\end{prop}
\begin{proof}
Given points $\xi,\eta$ in the component $Z\cu\overline X$, let $\pi_n$ be the gate-projection to an interval $I_n\cu X$ with ${\wh{\nu}\left((\s_{\xi}\triangle\s_{\eta})\setminus\mscr{H}(I_n))\right)\leq\frac{1}{n}}$; these exist for every $n\geq 1$, by Proposition~\ref{metric via intervals}. Since $X$ has compact intervals, every sequence in $I(\xi,\eta)$ has a subsequence $(\zeta_k)_{k\geq 0}$ with the property that $(\pi_n(\zeta_k))_{k\geq 0}$ converges for every $n\geq 1$. For all $k,h,n\geq 1$, we have ${d(\pi_n(\zeta_k),\pi_n(\zeta_h))\geq d(\zeta_k,\zeta_h)-\frac{1}{n}}$; thus, $(\zeta_k)_{k\geq 0}$ is Cauchy and compactness of $I(\xi,\eta)$ follows from Proposition~\ref{cc's complete}.
\end{proof}

The following should better justify the terminology introduced in Definition~\ref{connected components}.

\begin{prop}\label{sometimes cc's are connected}
If $X$ is connected, each component $Z\cu\overline X$ is connected. 
\end{prop}

Note however that $Z$ is not a connected component of $\overline X$ as the latter is connected, being the closure of $X$.

\begin{proof}
By Lemma~\ref{connected vs atoms} and Proposition~\ref{cc's complete}, it suffices to prove that no halfspace of $Z$ is an atom. Suppose for the sake of contradiction that there exists $\mf{k}\in\mscr{H}(Z)$ with $d(\mf{k},\mf{k}^*)>0$; let $(\xi,\eta)$ be a pair of gates for $(\mf{k},\mf{k}^*)$, as provided by Lemma~\ref{gates for pairs of convex sets 1}. Since $Z$ is convex in $\overline X$, the interval between $\xi$ and $\eta$ in $\overline X$ consists of the sole points $\xi$ and $\eta$. Let $\pi\colon\overline X\ra I(\xi,\eta)=\{\xi,\eta\}$ be the corresponding gate-projection; since $\pi$ is $1$-Lipschitz and $X$ is connected, we must have either $\pi(X)=\{\xi\}$ or $\pi(X)=\{\eta\}$. However, since $\xi\neq\eta$, there exists $\mf{h}\in\mscr{H}$ such that $\wt{\mf{h}}\in\mscr{H}(\xi|\eta)$; hence $\pi(\mf{h}^*)=\{\xi\}$ and $\pi(\mf{h})=\{\eta\}$, a contradiction.
\end{proof}

Observe that, if $Z$ is a component of the Roller boundary and $\mf{h}\in\mscr{H}$ is such that $\wt{\mf{h}}\cap Z$ and $\wt{\mf{h}}^*\cap Z$ are both nonempty, then they are halfspaces for the median-space structure of $Z$. The corresponding walls of $Z$ are enough to separate points in $Z$. However, not all walls of the median space $Z$ arise this way, essentially due to the fact that not every wall of the median algebra $\overline X$ arises from a wall of $X$. Still, \emph{almost every} wall of $Z$ comes from the above construction; see Proposition~\ref{pi_Z 2} below.

\begin{lem}\label{s_Z measurable}
Let $Z$ be a component of the Roller boundary. The sets $\s_Z:=\{\mf{h}\in\mscr{H}\mid Z\cu\wt{\mf{h}}\}$ and $\mscr{H}_Z:=\mscr{H}\setminus(\s_Z\sqcup\s_Z^*)$ are morally measurable.
\end{lem}
\begin{proof}
By part~1 of Lemma~\ref{gates vs inclusions}, we have $\s_Z\cap\mscr{H}(I)=\s_{\pi_I(Z)}\cap\mscr{H}(I)$ for every interval $I\cu X$ and $\pi_I(Z)$ is convex. The statement now follows from Lemma~\ref{adjacency}.
\end{proof}

Lemma~\ref{s_Z measurable} allows us to define gate-projections to boundary components. Namely, if $Z$ is a component of the Roller boundary, we have a morally measurable decomposition $\mscr{H}=\mscr{H}_Z\sqcup\s_Z\sqcup\s_Z^*$ and we can consider the map $\text{res}_Z\colon 2^{\mscr{H}}\ra 2^{\mscr{H}}$ that takes $E\cu\mscr{H}$ to $(E\cap\mscr{H}_Z)\sqcup\s_Z$. Using Lemma~\ref{s_Z measurable}, it is immediate to observe that $\text{res}_Z$ sends morally measurable ultrafilters to morally measurable ultrafilters and hence induces a map $\pi_Z\colon \overline X\ra \overline X$.

\begin{prop}\label{pi_Z 1}
The map $\pi_Z$ is the gate-projection to the closure of $Z$ in $\overline X$; this is canonically identified with the Roller compactification of $Z$ and will be denoted unambiguously by $\overline Z\cu\overline X$.
\end{prop}
\begin{proof}
By Proposition~\ref{retractions new}, the map $\pi_Z$ is a gate-projection to some gate-convex set $C\cu\overline X$. To avoid confusion, we denote by $W$ the closure of $Z$ in $\overline X$. Given $\xi\in\overline X$, we have $\pi_Z(\xi)=\xi$ if and only if $\wh{\nu}(\s_Z\setminus\s_{\xi})=0$; hence, $\pi_Z$ is the identity on $Z$ and, by part~1 of Lemma~\ref{compact median algebras}, it follows that $W\cu C$.

Suppose for the sake of contradiction that there exists $\xi\in C\setminus W$. By Lemma~\ref{gate-convexity vs compact intervals}, $W$ is gate-convex, so there exists a gate $\eta$ for $(\xi,W)$. Since $\xi\neq\eta$, we have $\pi_I(\xi)\neq\pi_I(\eta)$ for some interval $I\cu X$; every $\mf{h}\in\mscr{H}$ such that $\mf{h}\in\mscr{H}(\pi_I(\xi)|\pi_I(\eta))$ satisfies $\wt{\mf{h}}\in\mscr{H}(\xi|\eta)=\mscr{H}(\xi|W)$, hence $\mf{h}\in\s_Z\setminus\s_{\xi}$. This implies that $\wh{\nu}(\s_Z\setminus\s_{\xi})>0$, contradicting the fact that $\xi\in C$. 

We are left to identify $W$ with the Roller compactification $\overline Z$. We have a continuous morphism $f\colon\overline X\ra\overline Z$ mapping each $\xi\in\overline X$ to $(\pi_J(\xi))_{J\in\mscr{J}(Z)}$. Since $f(Z)=Z$, Lemma~\ref{M dense} and Proposition~\ref{cc's compact intervals} imply that $f$ is surjective. Moreover, $f(\xi)=f(\eta)$ if and only if no $\mf{h}\in\mscr{H}_Z$ satisfies $\wt{\mf{h}}\in\mscr{H}(\xi|\eta)$. If $\xi,\eta$ are distinct and lie in $W$, we have $0<d(\xi,\eta)=\wh{\nu}\left((\s_{\xi}\setminus\s_{\eta})\cap\mscr{H}_Z\right)$. Thus, the restriction of $f$ to $W$ is an isomorphism.
\end{proof}

In the rest of the section we will have to assume in addition that $X$ has finite rank; the necessity of this will be discussed below.

\begin{prop}\label{pi_Z 2}
Suppose $X$ is a complete, finite rank median space and let $Z$ be a component of $\partial X$. Then:
\begin{enumerate}
\item we have $\pi_Z(X)\cu Z$;
\item every thick halfspace of $Z$ is of the form $\wt{\mf{h}}\cap Z$ for a unique $\mf{h}\in\mscr{H}$;
\item we have $\text{rank}(Z)\leq\text{rank}(X)-1$.
\end{enumerate}
\end{prop}
\begin{proof}
Given $x\in X$ and $\xi\in Z$, let $\{\mf{h}_1,...,\mf{h}_k\}$ be a maximal set of pairwise-transverse halfspaces in $(\s_{\xi}\setminus\s_x)\cap\mscr{H}_Z$. We have:
\begin{align*}
d(\xi,\pi_Z(x))&=\wh{\nu}\left((\s_{\xi}\setminus\s_x)\cap\mscr{H}_Z\right) \\
&\leq\sum_{i=1}^k\wh{\nu}\left(\mscr{H}(x|\mf{h}_i)\right) + \sum_{i=1}^k\wh{\nu}\left(\mscr{H}(\mf{h}_i^*|\xi)\right) \\
&\leq\sum_{i=1}^kd(x,\mf{h}_i)+\sum_{i=1}^kd(\xi,\wt{\mf{h}}_i^*\cap Z)<+\infty.
\end{align*}
We now prove part~2. Given the partition $Z=\mf{k}\sqcup\mf{k}^*$ associated to a halfspace of $Z$, we obtain a partition of $X$ into the convex subsets $\pi_Z^{-1}(\mf{k})$ and $\pi_Z^{-1}(\mf{k}^*)$. These are halfspaces of $X$, unless $\pi_Z(X)\cu\mf{k}$ or $\pi_Z(X)\cu\mf{k}^*$. We show that this cannot happen if $\mf{k}$ is thick. Pick a point $\xi\in\mf{k}^*$ with $d(\xi,\mf{k})>0$; let $\eta$ be the gate for $(\xi,\overline{\mf{k}})$. Since $\xi\neq\eta$, there exists a halfspace $\mf{h}\in\mscr{H}$ with $\wt{\mf{h}}\in\mscr{H}(\xi|\eta)\cu\mscr{H}(\xi|\mf{k})$. Thus, $\pi_Z(x)\in\mf{k}^*$ for every $x\in\mf{h}^*$; in particular, $\pi_Z(X)\not\cu\mf{k}$. A symmetric argument shows that $\pi_Z(X)\not\cu\mf{k}^*$.

Finally, we prove part~3. Suppose for the sake of contradiction that $\text{rank}(Z)\geq r=\text{rank}(X)$. We have already observed that the halfspaces $\wt{\mf{h}}\cap Z$ with $\mf{h}\in\mscr{H}_Z$ are enough to separate points of $Z$. Lemma~\ref{rank with a subset} then shows that there exist $\mf{h}_1,...,\mf{h}_r\in\mscr{H}$ such that $\mf{k}_i:=\wt{\mf{h}_i}\cap Z\in\mscr{H}(Z)$ are pairwise transverse; in particular, $\mf{h}_1,...,\mf{h}_r$ must be pairwise transverse. Pick $x\in\mf{h}_1^*\cap...\cap\mf{h}_r^*$ and $\xi\in\mf{k}_1\cap...\cap\mf{k}_r$. Given two halfspaces $\mf{h},\mf{k}\in\s_{\xi}\setminus\s_x\cu\mscr{H}$, either $\mf{h}\cu\mf{k}$ or $\mf{k}\cu\mf{h}$ or they are transverse; observe that $\{\mf{h}_1,...,\mf{h}_r\}\cu\s_{\xi}\setminus\s_x$. If $\mf{h}\in(\s_{\xi}\setminus\s_x)\cap\s_Z$, we cannot have $\mf{h}\cu\mf{h}_i$ for any $i$ since $\s_Z$ is a filter and $\mf{h}_i\in\mscr{H}_Z$. Moreover, since $(\s_{\xi}\setminus\s_x)\cap\mscr{H}_Z$ has finite measure by part~1, the set $(\s_{\xi}\setminus\s_x)\cap\s_Z$ has infinite measure. Finally, the halfspaces $\mf{h}\in(\s_{\xi}\setminus\s_x)\cap\s_Z$ such that $\mf{h}_i\cu\mf{h}$ for some $i$ form a subset of finite measure, bounded above by the sum of the distances from $x$ to each $\mf{h}_i$; we conclude that there exists $\mf{h}\in\s_{\xi}\setminus\s_x$ that is transverse to $\mf{h}_1,...,\mf{h}_r$, a contradiction.
\end{proof}

Part~1 of Proposition~\ref{pi_Z 2} can fail without the finite rank assumption. Let $X$ be the $0$-skeleton of the ${\rm CAT}(0)$ cube complex whose vertex set is the restricted product $\{0,1\}^{(\N)}$ and whose edges join sequences with exactly one differing coordinate. Hyperplanes are in one-to-one correspondence with natural numbers and the Roller compactification $\overline X$ can be identified with the unrestricted product $\{0,1\}^{\N}$. Let $\xi\in\overline X$ be the point whose coordinates are all $1$; its component $Z$ consists of sequences with only finitely many zeroes. It is immediate to observe that $\overline Z=\overline X$; in particular, $\pi_Z$ is the identity on all of $\overline X$.

\begin{figure}
\centering
\includegraphics[width=3in]{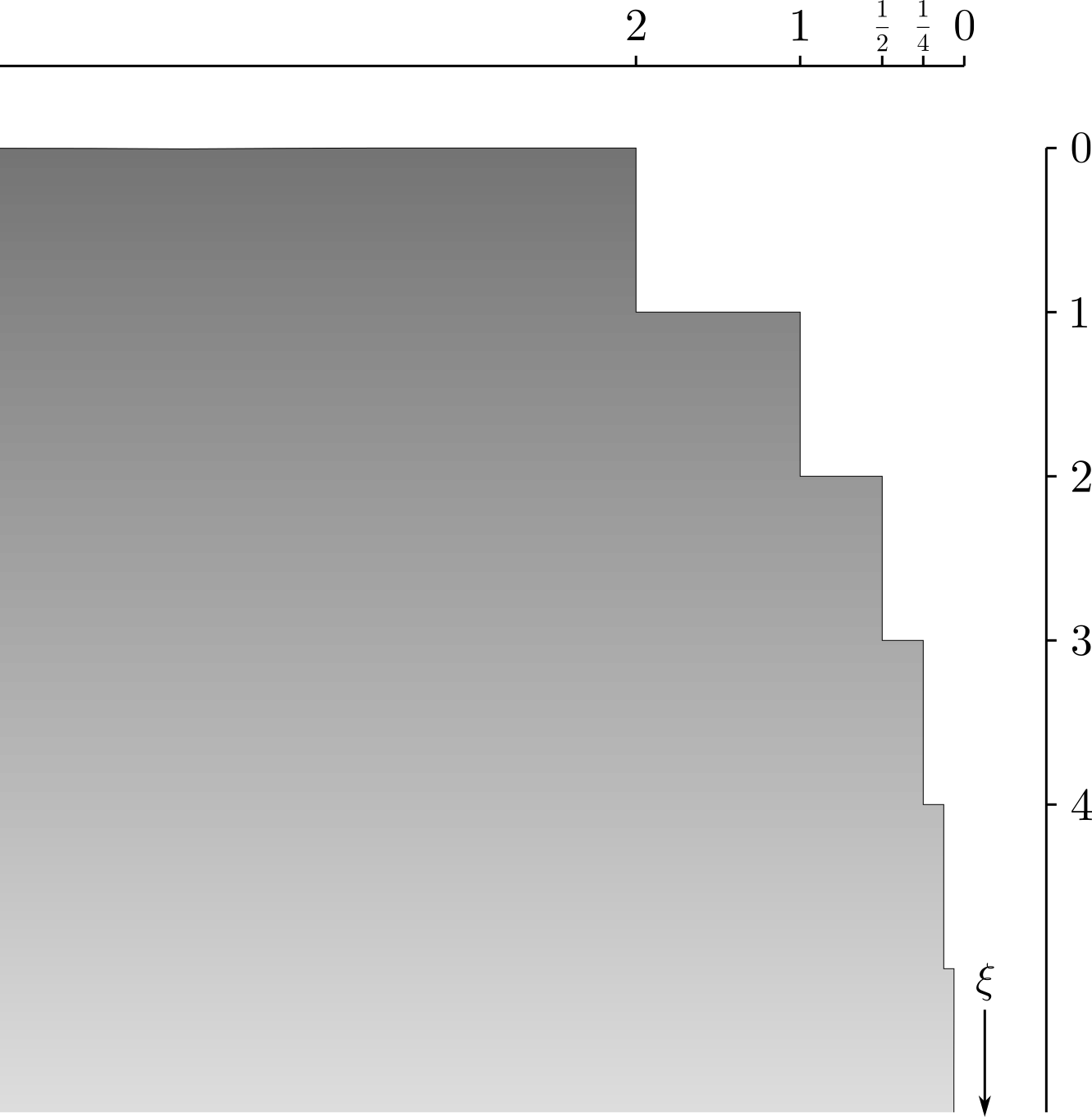}
\caption{}
\label{quadratic staircase}
\end{figure}

In general, even in finite rank, non-thick halfspaces of a component ${Z\cu\partial X}$ need not be of the form $\wt{\mf{h}}$ for some $\mf{h}\in\mscr{H}$. Consider the median space $X$ in Figure~\ref{quadratic staircase}. It is an infinite descending staircase with steps of constant height and exponentially-decreasing width; we consider $X$ as a subset of $\R^2$ with the restriction of the $\ell^1$ metric. It is a complete median space of rank two. Let $\xi\in\overline X$ be the point ``at the bottom'' of the staircase $X$ and let $Z\cu\overline X$ be its component of the Roller boundary. It is easy to notice that $\{\xi\}$ is a halfspace of $Z$, while, for every $\mf{h}\in\mscr{H}$, the set $\wt{\mf{h}}\cap Z$ either does not contain $\xi$ or contains a neighbourhood of $\xi$ in $Z$.

\begin{prop}\label{hamster}
Let $X$ be a complete finite rank median space with distinct components $Z_1,Z_2\cu\partial X$ satisfying ${\text{rank}(Z_1)=\text{rank}(Z_2)=k}$. There exists a component $W\cu\overline X$ such that $\text{rank}(W)\geq k+1$ and $W\cap I(\eta_1,\eta_2)\neq\emptyset$ for every $\eta_1\in Z_1$ and $\eta_2\in Z_2$.
\end{prop}
\begin{proof}
As in the proof of Proposition~\ref{pi_Z 2}, Lemma~\ref{rank with a subset} yields pairwise-trans\-verse halfspaces $\mf{h}_1,...,\mf{h}_k\in\mscr{H}_{Z_1}$. Suppose that $\mf{h}_i\in\mscr{H}_{Z_2}$ if and only if $i\leq s$, for some $0\leq s\leq k$. Similarly, let $\mf{k}_1,...,\mf{k}_k\in\mscr{H}_{Z_2}$ be pairwise transverse, with $\mf{k}_j\in\mscr{H}_{Z_1}$ if and only if $j\leq t$, for some $0\leq t\leq k$. 

Up to replacing some of these halfspaces with their complements, we can assume that $\mf{h}_i\cap\mf{k}_j\neq\emptyset$ and $\mf{h}_i^*\cap\mf{k}_j^*\neq\emptyset$ for all $1\leq i,j\leq k$ and, in addition, $\mf{k}_j^*\cap Z_1\neq\emptyset$ and $\mf{h}_i\cap Z_2\neq\emptyset$. This can be achieved as follows. We start by ensuring that $\mf{h}_i\in\s_{Z_2}$ and $\mf{k}_j\in\s_{Z_1}^*$ if $i>s$ and $j>t$. If $i\leq s$ and there exists $1\leq j\leq k$ such that $\mf{h}_i$ and $\mf{k}_j$ are not transverse, we pick the side of the wall $\{\mf{h}_i,\mf{h}_i^*\}$ that intersects both $\mf{k}_j$ and $\mf{k}_j^*$; since the $\mf{k}_h$ are pairwise transverse, they all determine the same side of $\{\mf{h}_i,\mf{h}_i^*\}$. Finally, we pick sides for the walls $\{\mf{k}_j,\mf{k}_j^*\}$ with $j\leq t$ in a similar way. Now, Helly's Theorem implies that there exist points
\begin{align*}
\xi_1&\in\wt{\mf{h}}_1^*\cap...\cap\wt{\mf{h}}_k^*\cap\wt{\mf{k}}_1^*\cap...\cap\wt{\mf{k}}_k^*\cap Z_1, \\
\xi_2&\in\wt{\mf{h}}_1\cap...\cap\wt{\mf{h}}_k\cap\wt{\mf{k}}_1\cap...\cap\wt{\mf{k}}_k\cap Z_2.
\end{align*}
The set $\s_{\xi_2}\setminus\s_{\xi_1}$ has infinite measure, since $Z_1$ and $Z_2$ are distinct. On the other hand, the sets $\s_{\mf{h}_i}\cap\s_{\xi_1}^*$, $\s_{\mf{k}_j}\cap\s_{\mf{h}_i^*}^*$ and $\s_{\xi_2}\cap\s_{\mf{k}_j^*}^*$ all have finite measure; indeed, $d(\xi_1,\wt{\mf{h}}_i\cap Z_1)$, $d(\mf{h}_i^*,\mf{k}_j)$ and $d(\xi_2,\wt{\mf{k}}_j^*\cap Z_2)$ are finite. We conclude that there exists $\mf{h}\in\s_{\xi_2}\setminus\s_{\xi_1}$ not lying in any of these sets; in particular, $\mf{h}$ is either transverse to all the $\mf{h}_i$ or it is transverse to all the $\mf{k}_j$. Without loss of generality, let us assume that we are in the former case. By Helly's Theorem, we can choose points $x_1\in\mf{h}_1^*\cap...\cap\mf{h}_k^*\cap\mf{h}^*$ and $x_2\in\mf{h}_1\cap...\cap\mf{h}_k\cap\mf{h}$; we set $\xi_1':=m(\xi_1,\xi_2,x_1)$, $\xi_2':=m(\xi_1,\xi_2,x_2)$. Observe that $\xi_1',\xi_2'$ belong to the interval $I(\xi_1,\xi_2)$ and, by Lemma~\ref{gate-projections are 1-Lip}, we have $d(\xi_1',\xi_2')\leq d(x_1,x_2)<+\infty$. In particular, $\xi_1'$ and $\xi_2'$ lie in the same component of $\overline X$, which we denote by $W$. Since $\wt{\mf{h}}_1,...,\wt{\mf{h}}_k,\wt{\mf{h}}$ all separate $x_1$ and $x_2$ and they intersect $I(\xi_1,\xi_2)$ nontrivially, they also separate $\xi_1'$ and $\xi_2'$. Hence, $\mf{h}_1,...,\mf{h}_k,\mf{h}$ all lie in $\mscr{H}_{W}$ and $\text{rank}(W)\geq k+1$.

Finally, $I(\eta_1,\eta_2)$ intersects $W$ for all $\eta_i\in Z_i$. Indeed, projecting $\xi_1'$ to $I(\eta_1,\eta_2)$ we only move it by a finite amount:
\begin{align*}
d\left(\xi_1',m(\eta_1,\eta_2,\xi_1')\right)&=d\left(m(\xi_1,\xi_2,\xi_1'),m(\eta_1,\eta_2,\xi_1')\right) \\
&\leq d\left(\xi_1,\eta_1\right)+d\left(\xi_2,\eta_2\right)<+\infty.
\end{align*}
\end{proof}

\begin{cor}\label{unique cc of max rk}
Every convex subset $C\cu\overline X$ intersects a unique component of $\overline X$ of maximal rank.
\end{cor}
\begin{proof}
Suppose $C$ intersects two distinct components $Z_1,Z_2$ of $\overline X$ of maximal rank. Given $\eta_1\in C\cap Z_1$ and $\eta_2\in C\cap Z_2$, we have $I(\eta_1,\eta_2)\cu C$ and this interval intersects a component of $\overline X$ of strictly higher rank by Proposition~\ref{hamster}, a contradiction.
\end{proof}

\bibliography{mybib}

\newcommand{\etalchar}[1]{$^{#1}$}
\begin{thebibliography}{BCG{\etalchar{+}}09}

\bibitem[BCG{\etalchar{+}}09]{BCGNW}
Jacek Brodzki, Sarah~J. Campbell, Erik Guentner, Graham~A. Niblo, and Nick~J.
  Wright.
\newblock Property {A} and {$\rm CAT(0)$} cube complexes.
\newblock {\em J. Funct. Anal.}, 256(5):1408--1431, 2009.

\bibitem[BH83]{Bandelt-Hedlikova}
Hans-J\"urgen Bandelt and Jarmila Hedl\'ikov\'a.
\newblock Median algebras.
\newblock {\em Discrete Math.}, 45(1):1--30, 1983.

\bibitem[BH99]{BH}
Martin~R. Bridson and Andr\'e Haefliger.
\newblock {\em Metric spaces of non-positive curvature}, volume 319 of {\em
  Grundlehren der Mathematischen Wissenschaften [Fundamental Principles of
  Mathematical Sciences]}.
\newblock Springer-Verlag, Berlin, 1999.

\bibitem[BM93]{Bandelt-Meletiou}
Hans-J\"urgen Bandelt and Gerasimos~C. Meletiou.
\newblock The zero-completion of a median algebra.
\newblock {\em Czechoslovak Math. J.}, 43(118)(3):409--417, 1993.

\bibitem[Bow13]{Bow1}
Brian~H. Bowditch.
\newblock Coarse median spaces and groups.
\newblock {\em Pacific J. Math.}, 261(1):53--93, 2013.

\bibitem[Bow14]{Bow2}
Brian~H. Bowditch.
\newblock Embedding median algebras in products of trees.
\newblock {\em Geom. Dedicata}, 170:157--176, 2014.

\bibitem[Bow16]{Bow4}
Brian~H. Bowditch.
\newblock Some properties of median metric spaces.
\newblock {\em Groups Geom. Dyn.}, 10(1):279--317, 2016.

\bibitem[CD17]{Chatterji-Drutu}
Indira Chatterji and Cornelia Dru\c{t}u.
\newblock Median geometry for spaces with measured walls and for groups.
\newblock {\em arXiv:1708.00254v1}, 2017.

\bibitem[CDH10]{CDH}
Indira Chatterji, Cornelia Dru\c{t}u, and Fr\'ed\'eric Haglund.
\newblock Kazhdan and {H}aagerup properties from the median viewpoint.
\newblock {\em Adv. Math.}, 225(2):882--921, 2010.

\bibitem[CFI16]{CFI}
Indira Chatterji, Talia Fern\'os, and Alessandra Iozzi.
\newblock The median class and superrigidity of actions on {$\rm CAT(0)$} cube
  complexes.
\newblock {\em J. Topol.}, 9(2):349--400, 2016.
\newblock With an appendix by Pierre-Emmanuel Caprace.

\bibitem[CL11]{Caprace-Lecureux}
Pierre-Emmanuel Caprace and Jean L\'ecureux.
\newblock Combinatorial and group-theoretic compactifications of buildings.
\newblock {\em Ann. Inst. Fourier (Grenoble)}, 61(2):619--672, 2011.

\bibitem[CMV04]{Cherix-Martin-Valette}
Pierre-Alain Cherix, Florian Martin, and Alain Valette.
\newblock Spaces with measured walls, the {H}aagerup property and property
  ({T}).
\newblock {\em Ergodic Theory Dynam. Systems}, 24(6):1895--1908, 2004.

\bibitem[CN05]{Chatterji-Niblo}
Indira Chatterji and Graham Niblo.
\newblock From wall spaces to {$\rm CAT(0)$} cube complexes.
\newblock {\em Internat. J. Algebra Comput.}, 15(5-6):875--885, 2005.

\bibitem[Cor15]{Cor2}
Yves Cornulier.
\newblock Irreducible lattices, invariant means, and commensurating actions.
\newblock {\em Math. Z.}, 279(1-2):1--26, 2015.

\bibitem[CS11]{CS}
Pierre-Emmanuel Caprace and Michah Sageev.
\newblock Rank rigidity for {${\rm CAT}(0)$} cube complexes.
\newblock {\em Geom. Funct. Anal.}, 21(4):851--891, 2011.

\bibitem[DCTV08]{Cornulier-Tessera-Valette}
Yves De~Cornulier, Romain Tessera, and Alain Valette.
\newblock Isometric group actions on {B}anach spaces and representations
  vanishing at infinity.
\newblock {\em Transform. Groups}, 13(1):125--147, 2008.

\bibitem[Dil50]{Dilworth}
Robert~P. Dilworth.
\newblock A decomposition theorem for partially ordered sets.
\newblock {\em Ann. of Math. (2)}, 51:161--166, 1950.

\bibitem[DK17]{DK}
Cornelia Drutu and Michael Kapovich.
\newblock Geometric group theory.
\newblock In {\em AMS series ``Colloquium Publications''}, 2017.

\bibitem[Fer18]{Fernos}
Talia Fern\'os.
\newblock The {F}urstenberg--{P}oisson boundary and {CAT}(0) cube complexes.
\newblock {\em Ergodic Theory Dynam. Systems}, 38(6):2180--2223, 2018.

\bibitem[Fio18]{Fioravanti2}
Elia Fioravanti.
\newblock The {T}its alternative for finite rank median spaces.
\newblock {\em Enseign. Math.}, 64(1-2):89--126, 2018.

\bibitem[Fio19]{Fioravanti3}
Elia Fioravanti.
\newblock Superrigidity of actions on finite rank median spaces.
\newblock {\em Adv. Math.}, 352:1206--1252, 2019.

\bibitem[FLM18]{Fernos-Lecureux-Matheus}
Talia Fern\'os, Jean L\'ecureux, and Fr\'ed\'eric Math\'eus.
\newblock Random walks and boundaries of {$\rm CAT(0)$} cubical complexes.
\newblock {\em Comment. Math. Helv.}, 93(2):291--333, 2018.

\bibitem[Gen16]{Genevois}
Anthony Genevois.
\newblock Contracting isometries of {${\rm CAT}(0)$} cube complexes and
  acylindrical hyperbolicity of diagram groups.
\newblock {\em arXiv:1610.07791v1}, 2016.

\bibitem[Gui05]{Guirardel}
Vincent Guirardel.
\newblock C\oe ur et nombre d'intersection pour les actions de groupes sur les
  arbres.
\newblock {\em Ann. Sci. \'Ecole Norm. Sup. (4)}, 38(6):847--888, 2005.

\bibitem[Gur05]{Guralnik}
Dan Guralnik.
\newblock Coarse decompositions of boundaries for {${\rm CAT}(0)$} groups.
\newblock {\em PhD Thesis, Technion, Haifa}, 2005.

\bibitem[Hag07]{Haglund}
Fr\'ed\'eric Haglund.
\newblock Isometries of {${\rm CAT}(0)$} cube complexes are semi-simple.
\newblock {\em arXiv:0705.3386v1}, 2007.

\bibitem[Hag13]{Hagen}
Mark~F. Hagen.
\newblock The simplicial boundary of a {${\rm CAT}(0)$} cube complex.
\newblock {\em Algebr. Geom. Topol.}, 13(3):1299--1367, 2013.

\bibitem[HP98]{Haglund-Paulin}
Fr\'ed\'eric Haglund and Fr\'ed\'eric Paulin.
\newblock Simplicit\'e de groupes d'automorphismes d'espaces \`a courbure
  n\'egative.
\newblock In {\em The {E}pstein birthday schrift}, volume~1 of {\em Geom.
  Topol. Monogr.}, pages 181--248. Geom. Topol. Publ., Coventry, 1998.

\bibitem[Isb80]{Isbell}
John~R. Isbell.
\newblock Median algebra.
\newblock {\em Trans. Amer. Math. Soc.}, 260(2):319--362, 1980.

\bibitem[Min16]{Minasyan}
Ashot Minasyan.
\newblock New examples of groups acting on real trees.
\newblock {\em J. Topol.}, 9(1):192--214, 2016.

\bibitem[Nic04]{Nica}
Bogdan Nica.
\newblock Cubulating spaces with walls.
\newblock {\em Algebr. Geom. Topol.}, 4:297--309, 2004.

\bibitem[Nic08]{Nica-thesis}
Bogdan Nica.
\newblock Group actions on median spaces.
\newblock {\em arXiv:0809.4099v1}, 2008.

\bibitem[NS13]{Nevo-Sageev}
Amos Nevo and Michah Sageev.
\newblock The {P}oisson boundary of {${\rm CAT}(0)$} cube complex groups.
\newblock {\em Groups Geom. Dyn.}, 7(3):653--695, 2013.

\bibitem[Rol98]{Roller}
Martin~A. Roller.
\newblock Poc sets, median algebras and group actions. {A}n extended study of
  {D}unwoody's construction and {S}ageev's theorem.
\newblock Preprint, University of Southampton, 1998.

\bibitem[Rud87]{Rudin}
Walter Rudin.
\newblock {\em Real and complex analysis}.
\newblock McGraw-Hill Book Co., New York, third edition, 1987.

\bibitem[Sag95]{Sageev}
Michah Sageev.
\newblock Ends of group pairs and non-positively curved cube complexes.
\newblock {\em Proc. London Math. Soc. (3)}, 71(3):585--617, 1995.

\bibitem[Sho52]{Sholander}
Marlow Sholander.
\newblock Trees, lattices, order, and betweenness.
\newblock {\em Proc. Amer. Math. Soc.}, 3:369--381, 1952.

\bibitem[vdV93]{Vel}
Marcel L.~J. van~de Vel.
\newblock {\em Theory of convex structures}, volume~50 of {\em North-Holland
  Mathematical Library}.
\newblock North-Holland Publishing Co., Amsterdam, 1993.

\bibitem[Ver93]{Verheul}
Eric~R. Verheul.
\newblock {\em Multimedians in metric and normed spaces}, volume~91 of {\em CWI
  Tract}.
\newblock Stichting Mathematisch Centrum, Centrum voor Wiskunde en Informatica,
  Amsterdam, 1993.

\bibitem[War58]{Ward}
Lewis~E. Ward, Jr.
\newblock On dendritic sets.
\newblock {\em Duke Math. J.}, 25:505--513, 1958.

\bibitem[Zei16]{Zeidler}
Rudolf Zeidler.
\newblock Coarse median structures and homomorphisms from {K}azhdan groups.
\newblock {\em Geom. Dedicata}, 180:49--68, 2016.

\end{thebibliography}
\bibliographystyle{alpha}

\end{document}